\documentclass[a4paper,12pt]{amsart}

\usepackage{amsmath, wrapfig}
\usepackage{dsfont, a4wide, amsthm, amssymb, amsfonts, graphicx}
\usepackage{xspace, psfrag, setspace, supertabular, color}

\begin{document}

\title{A new proof of the density Hales-Jewett theorem}

\author{D. H. J. Polymath}

\newtheorem{theorem}{Theorem}[section]
\newtheorem*{namedtheorem}{\theoremname}
\newcommand{\theoremname}{testing}
\newenvironment{named}[1]{\renewcommand{\theoremname}{#1} \begin{namedtheorem}}

\newtheorem{proposition}[theorem]{Proposition}
\newtheorem{lemma}[theorem]{Lemma}
\newtheorem{claim}[theorem]{Claim}
\newtheorem{corollary}[theorem]{Corollary}
\newtheorem{conjecture}[theorem]{Conjecture}
\newtheorem*{definition}{Definition}
\newtheorem{problem}[theorem]{Problem}
\newtheorem{example}[theorem]{Example}
\newtheorem{question}[theorem]{Question}
\newtheorem{remark}[theorem]{Remark}
\theoremstyle{plain}

\setlength{\textwidth}{6.5 in}
\setlength{\textheight}{9in}
\setlength{\oddsidemargin}{0in}
\setlength{\topmargin}{0in}
\addtolength{\voffset}{-.5in}

\def \url{}


\begin{abstract} The Hales--Jewett theorem asserts that for every $r$
and every $k$ there exists $n$ such that every $r$-colouring of the
$n$-dimensional grid $\{1, \dotsc, k\}^n$ contains a combinatorial
line. This result is a generalization of van der Waerden's theorem,
and it is one of the fundamental results of Ramsey theory. The theorem
of van der Waerden has a famous density version, conjectured by Erd\H
os and Tur\'an in 1936, proved by Szemer\'edi in 1975, and given a
different proof by Furstenberg in 1977. The Hales--Jewett theorem has
a density version as well, proved by Furstenberg and Katznelson in
1991 by means of a significant extension of the ergodic techniques
that had been pioneered by Furstenberg in his proof of Szemer\'edi's
theorem. In this paper, we give the first elementary proof of the
theorem of Furstenberg and Katznelson, and the first to provide a
quantitative bound on how large $n$ needs to be. In particular, we
show that a subset of $\{1,2,3\}^n$ of density $\delta$ contains a
combinatorial line if $n$ is at least as big as a tower of 2s of height
$O(1/\delta^2)$. Our proof is surprisingly simple: indeed, it gives
arguably the simplest known proof of Szemer\'edi's theorem.
\end{abstract}
\maketitle
\section{Introduction}

\subsection{Statement of our main result} \label{sec:basic-statements} 

The purpose of this paper is to give the first elementary proof of the
density Hales--Jewett theorem.  This theorem, first proved by
Furstenberg and Katznelson~\cite{FK89,FK91}, has the same relation to
the Hales--Jewett theorem~\cite{HJ63} as Szemer\'edi's
theorem~\cite{Sze75} has to van der Waerden's
theorem~\cite{vdW27}. Before we go any further, let us state all four
theorems. We shall use the notation $[k]$ to stand for the set
$\{1,2,\dotsc,k\}$. If $X$ is a set and $r$ is a positive integer,
then an $r$-\textit{colouring} of $X$ will mean a function
$\kappa\colon X\rightarrow [r]$. A subset $Y$ of $X$ is called
\textit{monochromatic} if $\kappa(y)$ is the same for every $y\in Y$.

We begin with van der Waerden's theorem.

\begin{theorem} \label{thm:vdw} 
For every pair of positive integers $k$ and $r$ there exists $N$ such that for
every $r$-colouring of $[N]$ there is a monochromatic arithmetic
progression of length $k$. 
\end{theorem}

Szemer\'edi's theorem is the \textit{density version} of van der
Waerden's theorem. That is, it says that in van der Waerden's theorem
one can always find an arithmetic progression in any colour class that 
is used reasonably often.

\begin{theorem} \label{thm:szem} 
For every
positive integer $k$ and every $\delta>0$ there exists $N$ such that
every subset $A\subseteq[N]$ of size at least $\delta N$ contains an
arithmetic progression of length $k$.  
\end{theorem}

The reason it is called a density version is that we think of $|A|/N$
as the density of $A$ inside $[N]$, so the condition on $A$ is that 
it has density at least $\delta$.

To state the Hales--Jewett theorem, we need a little more terminology.
The theorem is concerned with subsets of $[k]^n$, elements of which we
refer to as \emph{points} (or \emph{strings}). Instead of looking for
arithmetic progressions, the Hales--Jewett theorem looks for structures
known as \emph{combinatorial lines}. There are many equivalent
ways of defining these, of which one is the following. Let $[n]$ be
partitioned into sets $X_1,\dots,X_k,W$ in such a way that
$W$ is non-empty. Then take the set of all points $x$ such
that $x_i=j$ whenever $j\leq k$ and $i\in X_j$, and $x_i$ takes
the same value for every $i\in W$. The only choice we have
in specifying such an $x$ is the value we assign to the coordinates
$x_i$ with $i\in W$, so each line contains $k$ points.

Here is a simple example of a combinatorial line when $k=3$ and $n=8$:
\begin{equation*}
\{(\mathbf{1},3,\mathbf{1},2,2,\mathbf{1},1,2),(\mathbf{2},3,\mathbf{2},2,2,\mathbf{2},1,2),(\mathbf{3},3,\mathbf{3},2,2,\mathbf{3},1,2)\}
\end{equation*}
In this case the sets $X_1,X_2,X_3$ and $W$ are $\{7\}$, $\{4,5,8\}$,
$\{2\}$, and $\{1,3,6\}$, respectively.

The coordinates in $X_1\cup\dots\cup X_k$ are called the \textit{fixed
coordinates} of the line, and the coordinates in $W$ are the
\textit{variable coordinates} or \textit{wildcards}.

Another way of thinking of a line is as an element of the set $([k]\cup\{*\})^n$, 
where at least one coordinate takes the wildcard value $*$. To obtain the 
$k$ points in the line, one lets $j$ run from $1$ to $k$ and sets all
the wildcards equal to $j$. For instance, in this notation the line above is 
\begin{equation*}
(*,3,*,2,2,*,1,2).
\end{equation*}

With both these ways of thinking of combinatorial lines, it is clear that
there is a close relationship between lines in $[k]^n$ and points in 
$[k+1]^n$. Indeed, if one allows ``degenerate lines" in which the wildcard
sets are empty then there is an obvious one-to-one correspondence between
the two sets. This will be very important to us later.
\medskip

We are now ready to state the Hales--Jewett theorem.

\begin{theorem} \label{thm:hj} 
For every pair of
positive integers $k$ and $r$ there exists a positive number
HJ$(k,r)$ such that for every $n\geq$HJ$(k,r)$ and every 
$r$-colouring of the set $[k]^n$ there is a monochromatic 
combinatorial line.
\end{theorem} 

As with van der Waerden's theorem, we may consider the
density version of the Hales--Jewett theorem, where the density of $A
\subseteq [k]^n$ is $|A|/k^n$.  The following theorem was first
proved by Furstenberg and Katznelson~\cite{FK91}.

\begin{theorem} \label{thm:dhj} For every
positive integer $k$ and every real number $\delta>0$ there exists a positive
integer DHJ$(k,\delta)$ such that if $n\geq$DHJ$(k,\delta)$ and $A$ is
any subset of $[k]^n$ of density at least $\delta$, then $A$ contains a 
combinatorial line.  
\end{theorem}

We sometimes write ``DHJ$_k$'' to mean the $k$ case of this theorem.
The first nontrivial case, DHJ$_2$, is a weak version of Sperner's
theorem~\cite{Spe28}; we discuss this further in
Section~\ref{sec:sperner}.  We also remark that the Hales--Jewett
theorem easily implies van der Waerden's theorem, and likewise
for the density versions.  To see this, temporarily interpret $[m]$ as
$\{0, 1, \dotsc, m-1\}$ rather than $\{1, 2, \dotsc, m\}$, and
identify integers in $[N]$ with their base-$k$ representation in
$[k]^n$. It is then easy to see that a combinatorial line in $[k]^n$ 
corresponds to an arithmetic progression of length $k$ in $[N]$: 
if the wildcard set of the line is $S$, then the common difference 
of the progression is $\sum_{i \in S} k^{n-i}$. However, only 
very few arithmetic progressions of length $k$ in $[N]$ arise
in this way, so finding combinatorial lines is strictly harder than
finding arithmetic progressions. (Further evidence for this is
that several other results are easy consequences of the 
Hales-Jewett theorem and its density version: in particular,
it is an exercise to deduce the multidimensional Szemer\'edi
theorem from DHJ.)

In this paper, we give a new, elementary proof of the density
Hales--Jewett theorem, very different from that of Furstenberg
and Katznelson (though the discovery of one part of the argument,
sketched in \S \ref{partitioning}, was in part inspired by ergodic methods).
Our proof gives rise to the first known quantitative bounds for the theorem. 
Define the \textit{tower function} $T(n)$ inductively by taking
$T(1)=2$ and $T(n)=2^{T(n-1)}$ (so for instance $T(4)=2^{2^{2^2}}=65536$).
More generally, define (not quite standardly) the $k$th function $A_k$ 
in the Ackermann hierarchy by setting $A_k(1)=2$ and 
$A_k(n)=A_{k-1}(A_k(n-1))$, with $A_1(n)=2n$. Thus, the
$k$th function is obtained by iterating the $(k-1)$st function,
so $A_2(n)=2^n$ and $A_3(n)=T(n)$.

\begin{theorem} \label{thm:our-dhj} 
In the density Hales--Jewett theorem, one may take
DHJ$_{3}{\delta} = T(O(1/\delta^2))$.  For $k \geq 4$, the
bound DHJ$_{k}{\delta}$ we achieve is broadly comparable to
the function $A_k(1/\delta)$.
\end{theorem} 

By ``broadly comparable" we mean something like that it is 
much nearer to $A_k(1/\delta)$ than to $A_{k+1}(1/\delta)$. In
fact, the bound we obtain is something like $A_k(A_{k-1}(1/\delta))$.
(To give an idea, if we were to apply a composition of this kind to the 
function $A_{k-1}(n)=2^n$, then $A_k(n)$ would be a tower
of height $n$, whereas $A_k(A_{k-1}(n))$ would be a tower of 
height $2^n$.)

Another way of phrasing our result is in terms of the number
$c_{n,3}$, the cardinality of the largest subset of $[3]^n$ without a
combinatorial line.  Theorem~\ref{thm:our-dhj} states that
$c_{n,3}/3^n \leq O(1/\sqrt{\log^* n})$.  The only known lower bounds
appear in a parallel paper to this one that is by an overlapping set of
authors~\cite{Pol09}: in that paper it is shown that $c_{n,3} =
2,6,18,52,150,450$ for $n = 1,2,3,4,5,6$, and for large $n$ that
$c_{n,3}/3^n \geq \exp(-O(\sqrt{\log n}))$.  Generalizing to DHJ$_k$,
the authors show that $c_{n,k}/k^n \geq \exp(-O(\log n)^{1/\lceil
\log_2 k \rceil})$, using ideas from recent work on the construction
of Behrend~\cite{Beh46}.\\

\subsection{The motivation for finding a new proof.} 

Why is it interesting to give a new proof
of the density Hales--Jewett theorem? There are two main reasons. The
first is connected with the history of results and techniques in this
area. One of the main benefits of Furstenberg's proof of Szemer\'edi's
theorem was that it introduced a technique---ergodic methods---that
could be developed in many directions, which did not seem to be the
case with Szemer\'edi's proof. As a result, several far-reaching
generalizations of Szemer\'edi's theorem were
proved~\cite{BL96,FK78,FK85,FK91}, and for a long time nobody could
prove them in any other way than by using Furstenberg's methods. In
the last few years that has changed, and a programme has developed to
find new and finitary proofs of the results that were previously known
only by infinitary ergodic methods; see,
e.g.,~\cite{RS04,NRS06,RS06,RS07a,RS07b,Gow06,Gow07,Tao06,Tao07}. Giving
a non-ergodic proof of the density Hales--Jewett theorem was seen as a
key goal for this programme, especially since Furstenberg and
Katznelson's ergodic proof seemed significantly harder than the
ergodic proof of Szemer\'edi's theorem.  Having given a purely
finitary proof, we are able to obtain explicit bounds for how large
$n$ needs to be as a function of $\delta$ and $k$ in the density
Hales--Jewett theorem.  Such bounds could not be obtained via the
ergodic methods even in principle, since these proofs rely on the
Axiom of Choice.  Admittedly, our explicit bounds are not particularly
good: we start with a tower-type dependence for $k = 3$ and go up
a level of the Ackermann hierarchy each time we go from $k$ to 
$k+1$. However, they are in line with several other bounds in 
the area. For example, the best known bounds for the multidimensional 
Szemer\'edi theorem~\cite{Gow07,NRS06} (which is an easy consequence
of DHJ) are also of this type. 

A second reason that a new proof of the density Hales--Jewett theorem
is interesting is that it immediately implies Szemer\'edi's theorem,
and finding a new proof of Szemer\'edi's theorem seems always to be
illuminating---or at least this has been the case for the four main
approaches discovered so far (combinatorial~\cite{Sze75},
ergodic~\cite{Fur77,FKO82}, Fourier~\cite{Gow01}, hypergraph
removal~\cite{Gow06,Gow07,RS04,NRS06}). Surprisingly, in view
of the fact that DHJ is considerably more general than Szemer\'edi's 
theorem and the ergodic-theory proof of DHJ is considerably more
complicated than the ergodic-theory proof of Szemer\'edi's theorem,
the new proof we have discovered gives arguably the simplest proof
yet known of Szemer\'edi's theorem. It seems that by looking at a more general
problem we have removed some of the difficulty. Related to this is
another surprise. We started out by trying to prove the first
difficult case of the theorem, DHJ$_3$. The experience of all four of
the earlier proofs of Szemer\'edi's theorem has been that interesting
ideas are needed to prove results about progressions of length $3$,
but significant extra difficulties arise when one tries to generalize
an argument from the length-$3$ case to the general
case. Unexpectedly, it turned out that once we had proved the case
$k=3$ of the density Hales--Jewett theorem, it was straightforward to
generalize the argument to the $k \geq 4$ cases. We do not fully
understand why our proof should be different in this respect, but
it is perhaps a sign that the density Hales-Jewett theorem is at
a ``natural level of generality".

One might ask, if this is the case, why the proof of Furstenberg and
Katznelson seems to be \textit{more} complicated than the 
ergodic-theoretic proofs of Szemer\'edi's theorem and its
multidimensional version. An explanation for this discrepancy
is that our proof appears to be genuinely different from theirs (that
is, not just a translation of their proof into a more elementary language).
The clearest sign of this is that they use Carlson's theorem,
a powerful result in Ramsey theory, in an essential way, whereas
we have no need of any colouring results in our argument (unless
you count the occasional use of the pigeonhole principle).

Before we start working towards the proof of the theorem, we would
like briefly to mention that it was proved in a rather unusual ``open
source" way, which is why it is being published under a pseudonym. The
work was carried out by several researchers, who wrote their thoughts,
as they had them, in the form of blog comments at
\url{http://gowers.wordpress.com}.  Anybody who wanted to could
participate, and at all stages of the process the comments were fully
open to anybody who was interested.  (Indeed, taking some inspiration
from a few of these blog comments, Austin provided \emph{another} new
(ergodic) proof of the density Hales--Jewett theorem~\cite{Aus09}.)
This open process was in complete contrast to the usual way that
results are proved in private and presented in a finished form. The
blog comments are still available, so although this paper is a
polished account of the DHJ$_k$ argument, it is possible to read a
record of the entire thought process that led to the proof. The
constructions of new \emph{lower} bounds for the DHJ$_k$ problem,
mentioned in Section~\ref{sec:basic-statements}, are being published
by a partially overlapping set of researchers~\cite{Pol09}. The
participants in the project also created a wiki,
http://michaelnielsen.org/polymath1/, which contains sketches of
the arguments, links to the blog comments, and a great deal of related
material.

\subsection{Combinatorial subspaces and multidimensional DHJ}

We know from the density Hales-Jewett theorem that dense subsets of 
$[k]^n$ contain combinatorial lines. It is natural to wonder whether
there is a higher-dimensional version of this result, in which one
finds $d$-dimensional subspaces. Such a result does indeed exist,
and is a straightforward consequence of DHJ, as was observed
by Furstenberg and Katznelson. Since we shall need this extension,
we briefly define the relevant concepts and give the proof.

A $d$-dimensional combinatorial subspace is just like a combinatorial 
line except that there are $d$ wildcard sets instead of just one. In
other words we partition the ground set $[n]$ into $k+d$ sets
$X_1,\dots,X_k,W_1,\dots,W_d$ such that $W_1,\dots,W_d$ are
non-empty, and the subspace consists of all sequences $x$ such 
that $x_i=j$ whenever $i\in X_j$ and $x$ is constant on each set 
$W_r$. There is an obvious isomorphism between $[k]^d$ and any 
$d$-dimensional combinatorial subspace: the sequence
$z=(z_1,\dots,z_d)$ is sent to the sequence $x$ such that $x_i=j$ 
whenever $i\in X_j$ and $x_i=z_r$ whenever $x\in W_r$. 

Note that there is an obvious injection from the set of all $d$-dimensional
combinatorial subspaces of $[k]^n$ to $[k+d]^n$ (which becomes a
bijection if one allows the subspaces to be degenerate).

The multidimensional density Hales-Jewett theorem is the following.

\begin{theorem} \label{mdhj}
For every $\delta>0$ and every pair of integers $k$ and $d$ there exists
a positive integer MDHJ$(k,d,\delta)$ such that, for every $n\geq$MDHJ$(k,d,\delta)$
and every subset $A\subset[k]^n$, $A$ contains a $d$-dimensional
combinatorial subspace of $[k]^n$.
\end{theorem}

We shall refer to this theorem as MDHJ, and for each $k$ we shall refer
to the result for that $k$ as MDHJ$_k$.

\begin{proposition} \label{prop:mdhj-from-dhj} 
For every $k$, MDHJ$_k$ follows from DHJ$_k$. 
\end{proposition} 

\begin{proof} 
We prove the result by induction on $d$. Suppose we know MDHJ$_k$ for dimension 
$d-1$, and let $A \subseteq [k]^n$ have density at least $\delta$. Let 
$m = $MDHJ$(k,d-1,\delta/2)$, and write a typical string $z \in [k]^n$ as 
$(x,y)$, where $x \in [k]^m$ and $y \in [k]^{n-m}$. Call a string $y \in [k]^{n-m}$ 
``good if $A_y = \{x \in [k]^m : (x,y) \in A\}$ has density at least $\delta/2$ within 
$[k]^m$. Let $G \subseteq [k]^{n-m}$ be the set of good $y$'s. Then the density 
of $G$ within $[k]^{n-m}$ must be at least $\delta/2$, or $A$ could not have 
density at least $\delta$ in $[k]^{n-m}$.

By induction, for any good $y$ the set $A_y$ contains a $(d-1)$-dimensional 
combinatorial subspace. There are at most $M = (k+d-1)^m$ such subspaces,
because of the injection mentioned above. Therefore, there must be some
subspace $\sigma \subseteq [k]^m$ such that the set 
\[ G_\sigma = \{y \in [k]^{n-m} : (x,y) \in A\ \forall x \in \sigma\} \] 
has density at least $(\delta/2)/M$ within $[k]^{n-m}$. Provided that
$n \geq m + $DHJ$(k,\delta/2M)$, we may conclude from DHJ$_k$ that 
$G_\sigma$ contains a combinatorial line, $\lambda$. Then 
$\sigma \times \lambda$ is the desired $d$-dimensional subspace of
$[k]^n$ that is contained in $A$. 
\end{proof}

Because we have to iterate DHJ$_k$ with rapidly decreasing 
densities in order to obtain this result, the bound that we get from it
is very bad indeed: it is this that causes the Ackermann-type dependence
on $k$ in our main theorem.

\subsection{Density-increment strategies} \label{densityincrement}

Very briefly, our proof of DHJ$_k$
follows a \emph{density-increment strategy}, a technique that was pioneered by
Roth~\cite{Rot53} in his proof of the $k = 3$ case of Szemer\'edi's
theorem. There are now many such proofs in the literature, of which
most have the following form. One would like to prove that every dense
subset $A$ of a mathematical structure $S$ (such as an arithmetic progression
or the set $[k]^n$) contains a subset $X$ of a certain type (such
as a subprogression of length $k$ or a combinatorial line). It is usually hard
to show this in one step, so instead one proves that if $A$ has density
$\delta$ in $S$ and does \textit{not} contain a subset of the desired kind, 
then $S$ has a substructure $S'$ such that the density of $A$ inside
$S'$ is at least $\delta+c$, where $c$ is some positive constant that
depends only on $\delta$. This is the density increment. If $S'$ is 
of a similar nature to $S$, then one
can iterate this argument, and if $S$ is large enough, then one can
continue iterating until the density exceeds 1 and one has a contradiction,
from which one deduces that $A$ must after all contain a subset $X$ of
the desired kind.

Even getting directly from $S$ to a density increment on a substructure
$S'$ in one step is usually too hard, so typically there is an intermediate
stage. First, one finds a set $T$ that is in some sense ``simple" such
that the density of $A$ inside $T$ is at least $\delta+c$. Then one
proves that ``simple" sets $T$ can be partitioned into substructures 
$S_1,\dots,S_N$ and uses an averaging argument to show that
the density of $A$ inside some $S_i$ is also at least $\delta+c$.
There are also variants of this: for instance, it is enough to find 
subsets $S_1,\dots,S_N$ of $T$ such that every element of $T$ is 
in the same number of $S_i$, or even in approximately the same
number of $S_i$. 

A few proofs that have this basic structure are Roth's proof itself
(where the intermediate structure is a mod-$N$ arithmetic
progression, which can be partitioned into genuine arithmetic
progressions), Gowers's proof of Szemer\'edi's theorem~\cite{Gow01},
and an argument of Shkredov~\cite{Shk06b,Shk06a} that gives 
strong bounds for the ``corners problem", a result that we shall
discuss in detail in Section \ref{corners}.

\section{Sperner's theorem and its multidimensional version} \label{sec:sperner}

The case $k=2$ of the density Hales-Jewett theorem is equivalent
to the following statement: for every $\delta>0$ there exists $n$ 
such that if $\mathcal{A}$ is a collection of at least $\delta 2^n$
subsets of $[n]$ then there exist distinct sets $A,B\in\mathcal{A}$
such that $A\subset B$. The equivalence is easily seen if one
looks at the characteristic functions of the sets, in which case one
sees that a pair $(A,B)$ with $A\subset B$ corresponds to a
combinatorial line in $\{0,1\}^n$.

Exact bounds are known for this theorem. The nicest proof is the 
following one, which will have a considerable influence on our later 
proofs. Recall that an \textit{antichain} is a collection of sets such
that no set in the collection is a proper subset of any other.

\begin{theorem} \label{sperner}
Let $n$ be a positive integer and let $\mathcal{A}$ be an antichain
of subsets of $[n]$. Then $|\mathcal{A}|\leq\binom n{\lfloor n/2\rfloor}$.
\end{theorem}

\begin{proof}
Consider the following way of choosing a random subset of $[n]$.
One chooses a random permutation $\pi$ of $[n]$ and a random
integer $m\in\{0,1,\dots,n\}$ and takes the set $A=\{\pi(1),\dots,\pi(m)\}$.
Since $\mathcal{A}$ is an antichain, for each $\pi$ there is at most one 
$m$ such that the resulting set belongs to $\mathcal{A}$. Thus, the
probability of choosing a set in $\mathcal{A}$ is at most $1/(n+1)$.

Now the probability of choosing a particular set $A$ of size $m$ is
$(n+1)^{-1}\binom nm^{-1}$. Therefore, if we want $\mathcal{A}$
to be as large as possible but for the probability of choosing a
set in $\mathcal{A}$ to be at most $(n+1)^{-1}$, then we must 
choose $\mathcal{A}$ to consist of sets of size $m$ such that
$\binom nm$ is maximized. It follows that we cannot choose more
than $\binom n{\lfloor n/2\rfloor}$ sets, as claimed.
\end{proof}

We shall also need a multidimensional version of Sperner's
theorem. This time we are trying to maximize the size of $\mathcal{A}$
subject to the condition that it is not possible to find a $d$-dimensional
combinatorial subspace, which in set-theoretic terms means a 
collection of disjoint non-empty sets $A,A_1,\dots,A_d$ such that
$A\cup\bigcup_{i\in E}A_i\in\mathcal{A}$ for every 
$E\subset\{1,2,\dots,d\}$. The result we need was proved by
Gunderson, R\"odl and Sidorenko. However, for the convenience
of the reader we give a proof here, which is somewhat simpler
than theirs and gives a slightly better bound. (This improvement
has an imperceptible effect on our bound for DHJ$_3$ though.)

We begin with an easy and standard lemma. As usual, if $X$ is
a finite set and $Y$ is a subset of $X$, we write $\mu(Y)$ for
$|Y|/|X|$.

\begin{lemma} \label{intersections}
Let $X$ be a finite set and let $X_\gamma$ be a random subset
of $X$, where $\gamma$ is an element of a probability space
$\Gamma$. Suppose that $\mathbb{E}_\gamma\mu(X_\gamma)=\delta$.
Now let $\gamma$ and $\gamma'$ be chosen independently from
$\Gamma$. Then $\mathbb{E}_{\gamma,\gamma'}
\mu(X_\gamma\cap X_{\gamma'})\geq\delta^2$.
\end{lemma}

\begin{proof}
Let $\xi_\gamma$ be the characteristic function of $X_\gamma$. Then
\begin{align*}
\delta^2&=(\mathbb{E}_\gamma\mu(X_\gamma))^2\\
&=(\mathbb{E}_\gamma\mathbb{E}_x\xi_\gamma(x))^2\\
&\leq\mathbb{E}_x(\mathbb{E}_\gamma\xi_\gamma(x))^2\\
&=\mathbb{E}_x\mathbb{E}_{\gamma,\gamma'}\xi_\gamma(x)\xi_{\gamma'}(x)\\
&=\mathbb{E}_{\gamma,\gamma'}\mu(X_\gamma\cap X_{\gamma'}).\\
\end{align*}
The inequality above is Cauchy-Schwarz. The result follows.
\end{proof}

\begin{theorem} \label{multidimsperner}
Let $\mathcal{A}$ be a collection of subsets of $[n]$ that contains
no $d$-dimensional combinatorial subspace. Then the density of 
$\mathcal{A}$ is at most $(25/n)^{1/2^d}$.
\end{theorem}

\begin{proof}
Let $\delta$ be the density of $\mathcal{A}$. (That is, $\mathcal{A}$
has cardinality $\delta 2^n$.) For $i=1,2,\dots,d-1$ let $n_i=\lfloor n/4^{d-i}\rfloor$
and let $n_d=n-(n_1+\dots+n_{d-1})$. Note that $n_d\geq(2/3)n$.

Let us partition $[n]$ into sets
$J_1\cup\dots\cup J_{d-1}\cup E$ with $|J_i|=\lfloor n/4^{d-i}\rfloor$.
Note that $|E|\geq(2/3)n$.

Now consider the following way of choosing a random subset $A$ of 
$[n]$. First we choose a random permutation $\pi$ of $[n]$. Then 
we choose a random integer $s$ according to the binomial distribution
with parameters $n_1$ and $1/2$. Next, we let $B$ be a random
subset of $\{\pi(n_1+1),\dots,\pi(n)\}$. Finally, we let $A$ be the set
$\{\pi(1),\dots,\pi(s)\}\cup B$. The resulting distribution on $A$ is
uniform, as can be seen by conditioning on the set $\{\pi(1),\dots,\pi(n_1)\}$.

Let us write $A_{\pi,s}$ for the set $\{\pi(1),\dots,\pi(s)\}$ and $X_{\pi,s}$
for the set of all $B\subset\{\pi(n_1+1),\dots,\pi(n)\}$ such that 
$A_{\pi,s}\cup B\in\mathcal{A}$. Then the average density of 
$X_{\pi,s}$ (in the set of all subsets of $\{\pi(n_1+1),\dots,\pi(n)\}$)
is $\delta$. Therefore, by Lemma \ref{intersections}, if we first choose
$\pi$ randomly and then choose $s$ and $t$ independently at 
random from the binomial distribution as we did for $s$ above, then 
the average density of $X_{\pi,s}\cap X_{\pi,t}$ is at least $\delta^2$.

We would like $s$ and $t$ to be distinct. The probability that $s=t$
is equal to $2^{-n_1}\binom{2n_1}{n_1}$ (since it is the same as the probability
that $s+t=n_1$), which is well known to be at most $n_1^{-1/2}$, which
in turn is at most $2^{d-1}n^{-1/2}$. Therefore, the expected density
of $X_{\pi,s}\cap X_{\pi,t}$ conditional on $s\ne t$ is at least $\delta^2-2^{d-1}n^{-1/2}$.

Let us choose $s<t$ such that 
$\mu(X_{\pi,s}\cap X_{\pi,t})\geq\delta^2-2^{d-1}n^{-1/2}$, and let us
write $A_0^{(1)}$ and $A_1^{(1)}$ for $A_{\pi,s}$ and $A_{\pi,t}$. Note that 
$A_0^{(1)}$ is a proper subset of $A_1^{(1)}$, that both are disjoint from
the set $\{\pi(n_1+1),\dots,\pi(n)\}$ and that  $A_0^{(1)}\cup B$ and $A_1^{(1)}\cup B$
both belong to $\mathcal{A}$ for every $B\in X_{\pi,s}\cap X_{\pi,t}$.

Now let us run the argument again, with $n$ replaced by $n-n_1$, $n_1$
replaced by $n_2$ and $\mathcal{A}$ replaced by the set 
$\mathcal{A}_1=X_{\pi,s}\cap X_{\pi,t}$. It gives us sets 
$A_0^{(2)}$ and $A_1^{(2)}$ and a set $\mathcal{A}_2$ of subsets
of $\{\pi(n_2+1),\dots,\pi(n)\}$ such that $A_0^{(2)}$ is a proper subset 
of $A_1^{(2)}$, both $A_0^{(2)}$ and $A_1^{(2)}$ are disjoint from 
$\{\pi(n_2+1),\dots,\pi(n)\}$, both $A_0^{(2)}\cup B$ and $A_1^{(2)}\cup B$ belong
to $\mathcal{A}_1$ for every $B\in\mathcal{A}_2$, and the density of
$\mathcal{A}_2$ is at least
\begin{equation*}
(\delta^2-2^{d-1}n^{-1/2})^2-2^{d-2}n^{-1/2}\geq\delta^4-2^{d-1}n^{-1/2},
\end{equation*}
where for the last inequality we used the fact that $\delta<1/2$.

If we continue this process and have shown that $\mathcal{A}_r$ has
density at least $\delta^{2^r}-2^{d-r+1}n^{-1/2}$, then at the next stage
we obtain $\mathcal{A}_{r+1}$ with density at least
\begin{equation*}
(\delta^{2^r}-2^{d-r+1}n^{-1/2})^2-2^{d-r-1}n^{-1/2}\geq\delta^{2^{r+1}}-2^{d-r}n^{-1/2}.
\end{equation*}
Therefore, as long as $\delta^{2^{d-1}}-4n^{-1/2}\geq 1/2\sqrt{2n/3}$, then
by Sperner's theorem $\mathcal{A}_{d-1}$ contains two sets $A_d^{(0)}$ and
$A_d^{(1)}$, with $A_d^{(0)}$ a proper subset of $A_d^{(1)}$. This gives us the 
desired combinatorial subspace (which consists of all sets of the form
$A_1^{(\epsilon_1)}\cup\dots\cup A_d^{(\epsilon_d)}$ such that
each $\epsilon_i$ is either 0 or 1. 

The inequality we need is true if $n\geq 5^2/\delta^{2^d}$, so the theorem
is proved.
\end{proof}

\section{Equal-slices measure and probabilistic DHJ}

The proof of Sperner's theorem can be regarded as follows. First,
one chooses a different measure on the power set of $[n]$, where
to choose a set you first choose its cardinality $m$ uniformly at random
from $\{0,1,2,\dots,n\}$ and you then choose a random set of size $m$.
The set of all subsets of $[n]$ of size $m$ is sometimes denoted
by $[n]^{(m)}$ and called a \textit{layer} or \textit{slice} of the cube.
We therefore call the resulting probability measure on the power
set of $[n]$, or equivalently on $[2]^n$, the \textit{equal-slices
measure}.

This measure arises so naturally in the averaging argument that
we used to prove Sperner's theorem that it is tempting to say that
the ``real" theorem is that the maximum possible equal-slices 
measure of an antichain is $1/(n+1)$. One then converts that
into a slightly artificial (and weaker) statement about the uniform 
measure.

The advantage of equal-slices measure is not just cosmetic,
however: it and its obvious generalization to $[k]^n$ will play
a crucial role in our proof. Rather than saying straight away why
this should be, we shall prove a result using equal-slices measure 
and explain why it would be problematic to give a uniform version.

But before we do that, let us give a formal definition of the equal-slices
measure on $[k]^n$. This time we choose, uniformly at random from
all possibilities, a $k$-tuple $(a_1,\dots,a_k)$ of non-negative integers
that add up to $n$, and then we choose a sequence $x\in[k]^n$ such that 
for each $j$ the set $X_j=\{i:x_i=j\}$ has cardinality $a_j$, again 
uniformly from all possibilities (of which there are $\binom n{a_1,\dots,a_k}$).  

The number of slices can be worked out by a ``holes and pegs"
argument: given any subset $B=\{b_1,\dots,b_{k-1}\}$ of 
$\{1,2,\dots,n+k-1\}$ of size $k-1$, let $a_i$ be the number
of integers strictly between $b_{i-1}$ and $b_i$, where we
treat $b_0$ as $0$ and $b_k$ as $n+k$. This gives us all
possible sequences $(a_1,\dots,a_k)$ exactly once each,
so the number of slices is $\binom{n+k-1}{k-1}$. 

For use in the proof of the next theorem,
we note that if $k=3$ then the number of slices with $a_2=0$
is $n+1$, so the probability that $a_2=0$ is $(n+1)/\binom{n+2}2$,
which equals $2/(n+2)$. 

We can easily define equal-slices measure 
for combinatorial lines as well. Indeed, there is a one-to-one correspondence
between lines in $[k]^n$ and points in $[k+1]^n$, at least if one
allows the lines to be degenerate. If $y\in[k+1]^n$, then the corresponding
line consists of all points of the form $y^{k+1\rightarrow j}$ with
$j\in[k]$; in other words, the set of $i$ such that $y_i=k+1$ is
treated as a wildcard set.

\subsection{A probabilistic version of Sperner's theorem}

As mentioned in the introduction, our proof of DHJ uses a density-increment
strategy: that is, we assume that $A$ does not contain a line and deduce
that $A$ has increased density inside some subspace. In almost all known 
proofs of this kind, one can in fact get away with a weaker hypothesis. If
$A$ is a dense set inside which one wishes to find some structure, then one 
can find a density increment on the assumption that $A$ has ``too few" subsets
of the kind one is looking for, or more generally ``the wrong number" of such 
subsets, where ``the right number" is the number you would expect if $A$ is
a random subset of density $\delta$. Similarly, it is also possible to find equivalent 
versions of the theorems that say that a set $A$ of density $\delta$ contains not 
just one subset of the desired kind, but ``many" such subsets, where this means 
that if you choose a random such subset then with probability at least $c=c(\delta)>0$
it will lie in $A$. A statement like this is called a ``probabilistic version" of the 
density theorem.

This is a sufficiently important feature of previously known arguments that
it is initially unsettling to observe that it is false for DHJ even when $k=2$.
The reason is a simple one. By standard measure-concentration results,
almost all points in $[2]^n$ have roughly $n/2$ 1s and $n/2$ 2s. By the
same results, almost all combinatorial lines have roughly $n/3$ fixed 1s, 
$n/3$ fixed 2s and $n/3$ variable coordinates. (A precise statement 
expressing this can be found in Lemma \ref{balanced} below.) It follows 
that there is a set of density almost 1 (the set of sequences with roughly 
equal numbers of 1s and 2s) that contains only a tiny fraction of all lines 
(ones with roughly $n/2$ fixed 1s, roughly $n/2$ fixed 2s and a very small
wildcard set). 

However, this does not mean that there is no probabilistic version of
DHJ, which is fortunate as we shall need one later. It merely means
that the uniform measure is the wrong measure in which to express it.
To illustrate this point, we now prove a ``probabilistic" version of DHJ$_2$. 
It tells us that an equal-slices-dense subset of $[2]^n$ must contain an 
equal-slices-dense set of lines. 

\begin{theorem} \label{probabilisticsperner}
Let $A$ be a subset of $[2]^n$ of equal-slices density $\delta$. Then
the set of (possibly degenerate) combinatorial lines in $A$ has equal-slices 
density at least $\delta^2(n+1)/(n+2)$.
\end{theorem}

\begin{proof}
Let $\pi$ be a random permutation of $[n]$ and let $s$ and $t$ be elements
of $\{0,1,2,\dots,n\}$ chosen independently and uniformly at random. Let
us write $x_{\pi,m}$ for the sequence that takes the value $1$ at $\pi(1),\dots,\pi(m)$
and $0$ everywhere else, and let $X_\pi$ be the
number of the sequences $x_{\pi,s}$ that belong to $A$. Then 
$\mathbb{E}X_\pi=\delta n$, by the definition of equal-slices measure.

From this it follows that $\mathbb{E}X_\pi^2$ is
at least $\delta^2 n^2$. But $X_\pi^2$ is the number of pairs $(s,t)$ such that
both $x_{\pi,s}$ and $x_{\pi,t}$ belong to $A$. Therefore, if we choose 
a random pair $\{x_{\pi,s},x_{\pi,t}\}$ then the probability that both its
constituent sequences belong to $A$ is at least $\delta^2$. 

Now each such pair forms a combinatorial line. 
If $s\leq t$, then this line consists of all sequences $x$ such that $x_i=1$ if
$i\in\{\pi(1),\dots,\pi(s)\}$, $x_i=0$ if $i\in\{\pi(t+1),\dots,\pi(n)\}$, 
and $x$ is constant on the set $\{\pi(s+1),\dots,\pi(t)\}$. (Thus, the
set $\{\pi(s+1),\dots,\pi(t)\}$ is the wildcard set.) If $t\leq s$ then 
we simply interchange the roles of $s$ and $t$ in the above. (If
$t=s$ then we have a degenerate line and interchanging the roles
of $s$ and $t$ makes no difference.) 

There is one technical detail that we need to address, which is that
the probability $p(\ell)$ that we choose a particular combinatorial line 
$\ell$ is not quite the equal-slices probability $q(\ell)$. In particular, the 
probability that the line is degenerate is $(n+1)^{-1}$ instead of $2(n+2)^{-1}$.
However, if we condition on the event that $s\ne t$, then 
we are choosing a random subset of $\{0,1,2,\dots,n\}$ of size 2,
and such pairs are in one-to-one correspondence with triples
$(a_1,a_2,a_3)$ such that $a_1+a_2+a_3=n$ and $a_2\ne 0$.
Thus, $p(\ell)=(n+2)q(\ell)/2(n+1)$ if $\ell$ is degenerate, and
$p(\ell)=(1-(n+1)^{-1})q(\ell)/(1-2(n+2)^{-1})=(n+2)q(\ell)/(n+1)$
if $\ell$ is non-degenerate.

From the above calculation it follows that the set of lines in $A$ has 
equal-slices density at least $\delta^2(n+1)/(n+2)$, as claimed.
\end{proof}

The equal-slices density of the set of degenerate lines is $O(n^{-1})$, so
this result implies that there is a dense set of non-degenerate
combinatorial lines in $A$ as well.

\subsection{Non-degenerate equal-slices measure}

For technical reasons, it is sometimes convenient, when talking
about equal-slices measure, to condition on the event that every
$j\in[k]$ is equal to $x_i$ for some $i$. Indeed, we have already
seen in the proof of Theorem \ref{probabilisticsperner} that 
degenerate slices---that is, slices for which this condition does
not hold---can be slightly problematic. It turns out that if we
condition on the slices not being degenerate, then we can prove
a useful lemma that would hold only approximately, and after
tedious consideration of the degenerate cases, if we used the
equal-slices measure itself.

Let us therefore define the \textit{non-degenerate equal-slices
measure} on $[k]^n$ as follows. One first chooses a random
$k$-tuple of positive (rather than non-negative) integers $(a_1,\dots,a_k)$
that add up to $n$ and then a random sequence $x\in[k]^n$ such
that $|X_j|=a_j$ for each $j$, where as before $X_j$ is the set
$\{i\in[n]:x_i=j\}$. 

A helpful equivalent way of defining this measure is as follows. To select
a random point $x\in[k]^n$, one places $n$ points $q_1,\dots,q_n$
around a circle in a random order. That creates $n$ gaps between
consecutive points. One chooses a random set of $k$ of these
gaps and places further points $r_1,\dots,r_k$ into the gaps,
again in a random order. Finally, one sets $x_i$ to be $j$ if
and only if $r_j$ is the first point out of $r_1,\dots,r_k$ that
you come to if you go round clockwise starting at $q_i$. 

Note that since the $q_i$ are in a random order, precisely 
the same distribution will arise if the $r_j$ are placed in some
fixed order rather than their order too being randomized. 
However, it is more convenient to randomize everything.
Note also that since we do not allow two different $r_j$
to occupy the same gap, for each $j$ there exists $i$
such that $x_i=j$. Finally, note that apart from this constraint,
all slices are equally likely. Therefore, we really do have
the equal-slices measure conditioned on the event that
the slices are non-degenerate.

To see the effect that this conditioning has, let us give an upper
bound for the probability is that a slice \textit{is} degenerate.

\begin{lemma}\label{degenerate}
Let $x$ be an equal-slices random point of $[k]^n$. Then the probability
that no coordinate of $x$ is equal to $k$ is $\frac{k-1}{n+k-1}$. In particular,
it is at most $k/n$.
\end{lemma}

\begin{proof}
To choose $k$ non-negative integers $a_1,\dots,a_k$ that add up to $n$, 
and to do so uniformly from all possibilities, one can choose a random subset 
$P=\{p_1<\dots<p_{k-1}\}\subset\{1,2,\dots,n+k-1\}$ of size $k-1$ (of ``pegs") 
and let $a_i$ be the number of integers strictly between $p_{i-1}$ and $p_i$,
where we set $p_0=0$ and $p_k=n+k$. The probability that no coordinate of
$x$ is equal to $k$ is the probability that $a_k=0$, which is the
probability that $n+k-1\in P$, which is $\frac{k-1}{n+k-1}$, as claimed.
\end{proof}

\begin{corollary} \label{nondegeneracy}
Let $\nu$ and $\tilde\nu$ be the equal-slices and non-degenerate equal-slices
measures on $[k]^n$, respectively. Then for any set $A\subset[k]^n$ we
have $|\nu(A)-\tilde\nu(A)|\leq k^2/n$
\end{corollary}

\begin{proof}
It follows from Lemma \ref{degenerate} that the probability that a slice is 
degenerate is at most $k^2/n$. Therefore, if $A$ is a set that consists only 
of non-degenerate sequences, then its non-degenerate equal-slices 
measure is $(1-c)^{-1}$ times its equal-slices measure, for some $c<k^2/n$. 
Therefore, for such a set, $0\leq\tilde\nu(A)-\nu(A)=c\tilde\nu(A)\leq k^2/n$.
If $A$ consists only of degenerate sequences, then $0\leq\nu(A)-\tilde\nu(A)
=\nu(A)\leq k^2/n$. The result follows, since if one takes a union of sets of
the two different kinds, then the differences cancel out rather than reinforcing
each other.
\end{proof}

For later use, we slightly generalize Lemma \ref{degenerate}.

\begin{lemma} \label{knottoosmall}
Let $x$ be chosen randomly from $[k]^n$ using the equal-slices distribution.
Then the probability that fewer than $m$ coordinates of $x$ are equal to k is at
most $mk/n$.
\end{lemma}

\begin{proof}
Let $P$ be as in the proof of Lemma \ref{degenerate}. This time we are
interested in the probability that $p_{k-1}\geq n+k-m$. The number with 
$p_{k-1}=n+k-s$ is $\binom{n+k-s-1}{k-2}$, which is at most $\binom{n+k-2}{k-2}$,
which as we noted in the proof of Lemma \ref{degenerate} is at most
$\frac kn\binom{n+k-1}{k-1}$. The result follows.
\end{proof}

\begin{corollary} \label{nottooimbalanced}
Let $x$ be chosen randomly from $[k]^n$ using the equal-slices distribution.
Then the probability that there exists $j\in[k]$ such that fewer than $m$ 
coordinates of $x$ are equal to j is at most $mk^2/n$.
\end{corollary}

\begin{proof}
This follows immediately from Lemma \ref{knottoosmall}.
\end{proof}

Now let us return to our discussion of the non-degenerate equal-slices
measure. The next result tells us that it has a beautiful property. Let us
use the expression $\tilde\nu$-random to mean ``random and chosen
according to the non-degenerate equal-slices measure". Then the 
property is that a $\tilde\nu$-random point in a $\tilde\nu$-random 
subspace with no fixed coordinates is a $\tilde\nu$-random point. This 
result will enable us to carry out clean averaging arguments when we 
are using equal-slices measure.

We have not said what we mean by a $\tilde\nu$-random subspace 
with no fixed coordinates, but the definition is a straightforward modification 
of our earlier definition of the equal-slices density of a set of combinatorial lines.
First, a $d$-dimensional subspace with no fixed coordinates is simply
a subspace obtained by partitioning $[n]$ into $d$ non-empty sets 
$X_1,\dots,X_d$ and taking the set of all sequences $x\in[k]^n$ that 
are constant on each $X_i$. For brevity, let us call these \textit{special}
subspaces.

As we mentioned earlier, just as a combinatorial line in $[k]^n$ can be associated with a
point in $[k+1]^n$, so a $d$-dimensional combinatorial subspace
in $[k]^n$ can be associated with a point in $[k+d]^n$. If the subspace
is special, then it will in fact be associated with a point in $[d]^n$.

In the reverse direction, if $x\in[k+d]^n$, then the corresponding $d$-dimensional
subspace is the set of all points $y$ such that $y_i=j$ whenever
$j\in[k]$ and $x_i=j$, and $y$ is constant on all sets of the form
$X_j=\{i:x_i=j\}$ when $j>k$. Thus, the wildcard sets are the $d$ sets
$X_{k+1},\dots,X_{k+d}$. In the case of special subspaces, we 
take instead $x$ to belong to $[d]^n$ and the wildcard sets 
are $X_1,\dots,X_d$. 

Therefore, when we talk about the equal-slices measure or
non-degenerate equal-slices measure of a set of special 
$d$-dimensional subspaces, we are associating with each subspace
a point in $[d]^n$ and taking the corresponding measure there.
(A small detail is that for this to work we need the wildcard sets
in the combinatorial subspace to form a sequence rather than
just a set. In other words, if we permute the ``basis" then we
are considering the result as a different subspace, even though
it consists of the same points. Alternatively, one could regard
the correspondence as being $d!$-to-one.)

\begin{lemma} \label{equalslicestoequalslices}
Let $n$, $k$ and $d$ be positive integers with $n\geq k+d$. Suppose
that a point $x\in[k]^n$ is chosen randomly by first choosing a 
$\tilde\nu$-random special $d$-dimensional subspace $V$ of $[k]^n$ and
then choosing a $\tilde\nu$-random point in $V$. Then the resulting
distribution is the non-degenerate equal-slices measure on $[k]^n$.
\end{lemma}

\begin{proof}
To prove this we use the second method of defining the non-degenerate
equal-slices measure. That is, we choose a random subspace as
follows. First, we place $n$ points $q_1,\dots,q_n$ in a random order
around a circle. Next, we choose $d$ points $r_1,\dots,r_d$ and place
them in random gaps between the $q_i$, with no two of the $r_h$
occupying the same gap. Then the wildcard set $X_h$ will consist
of all $h$ such that $r_h$ is the first of the points $r_1,\dots,r_d$ if
you go clockwise round the circle from $q_i$. Let us call the
set of points $q_i$ with this property, together with $r_h$, the 
$h$th \textit{block}.

How do we then choose a random point $x$ in this subspace? We
can think of it as follows. We take the $d$ blocks and randomly 
permute them. We then randomly place $k$ points $s_1,\dots,s_k$
in gaps between blocks (with no two $s_j$ in the same gap). 
Then $x_i=j$ if $s_j$ is the first of the points $s_1,\dots,s_k$ if
you go clockwise round from $q_i$ (after the blocks have been
permuted). 

Now consider a second way of choosing a random point in 
$[k]^n$. We proceed exactly as above, except that this time we
do not bother to permute the blocks. We claim that this
gives rise to exactly the same distribution.

To see this, let us call two valid arrangements of the points
$q_1,\dots,q_n$ and $r_1,\dots,r_d$ \textit{equivalent} if
one is obtained from the other by a permutation of the blocks.
Then all the equivalence classes have size $d!$, so randomly
choosing an arrangement is the same as randomly choosing
an arrangement and then randomly changing it to an equivalent
arrangement.

Now the second way of choosing a random sequence amounts to
choosing the random points $q_1,\dots,q_n$ and $r_1,\dots,r_d$,
randomly choosing $k$ of the points $r_1,\dots,r_d$ and calling
them $s_1,\dots,s_k$ (in a random order) and finally using the 
points $q_1,\dots,q_n,s_1,\dots,s_k$ to define a point in $[k]^n$
in the usual way. But this is precisely the non-degenerate 
equal-slices measure on $[k]^n$.
\end{proof}

\subsection{A probabilistic version of the density Hales-Jewett theorem}

With the help of Corollary \ref{nondegeneracy} and 
Lemma \ref{equalslicestoequalslices}, it is straightforward to prove
that a probabilistic version of DHJ$_k$ follows from an
``equal-slices version". Let us begin by stating the equal-slices 
version.

\begin{theorem} \label{equalslicesdhj}
For every $\delta>0$ and every positive integer $k$ there exists
$n$ such that every set $A\subset[k]^n$ of equal-slices density
at least $\delta$ contains a combinatorial line.
\end{theorem}

We shall show later that Theorem \ref{equalslicesdhj} follows from
DHJ$_k$ itself. For now let us assume it and deduce a probabilistic
version. We shall write EDHJ$(k,\delta)$ for the smallest integer $m$
such that every subset $A\subset[k]^m$ of equal-slices density
at least $\delta$ contains a combinatorial line.

\begin{theorem} \label{probabilisticdhj}
Let $\delta>0$ and let $k$ be an integer greater than or equal to $2$.
Then there exists $\theta=$PDHJ$(k,\delta)$ such that for every 
$n\geq\max\{m,4k^2/\delta\}$ and every $A\subset[k]^n$ of
equal-slices density at least $\delta$ the equal-slices density of
the set of combinatorial lines in $A$ is at least $\theta$. Moreover,
if $m=$EDHJ$(k,\delta/4)$ then we can take $\theta=(\delta/9)(k+1)^{-m}$. 
\end{theorem}

\begin{proof} (Assuming Theorem \ref{equalslicesdhj}.) By Corollary
\ref{nondegeneracy} the non-degenerate equal-slices density
$\tilde\nu(A)$ of $A$ is at least $\delta-k^2/n$. Since
$n\geq 4k^2/\delta$, this is at least $3\delta/4$. 

Let $V$ be a random $m$-dimensional
special subspace of $[k]^n$, chosen according to the non-degenerate equal-slices 
measure. Then Lemma \ref{equalslicestoequalslices} implies that
the expected non-degenerate equal-slices density of $A$ inside
$V$ is also at least $3\delta/4$, from which it follows that with probability
at least $\delta/4$ this density is at least $\delta/2$. 

Let $V$ be a subspace inside which $A$ has non-degenerate equal-slices
density at least $\delta/2$. Remove from $A\cap V$ all degenerate
strings. The resulting set $A'\cap V$ still has density at least $\delta/2$.
By Corollary \ref{nondegeneracy} again, this implies that the equal-slices
density of $A'$ inside $V$ is at least $\delta/4$.   

But by our choice of $m$ this means that with probability at least
$\delta/4$ the set $A'\cap V$ contains a combinatorial line. Moreover,
since $A'\cap V$ contains no degenerate strings, this line must have
fixed coordinates of every single value. 

The number of such lines is at most $(k+1)^m$. Therefore, if you choose a 
random special subspace and inside it you choose a line according
to the non-degenerate equal-slices measure, then with probability
at least $(\delta/4)(k+1)^{-m}$ it will be a line in $A$.

But by Lemma \ref{equalslicestoequalslices} the way we have just chosen 
this line was according to the non-degenerate equal-slices measure. By the 
proof of Corollary \ref{nondegeneracy}, the equal-slices probability is at least
$(\delta/4)(k+1)^{-m}(1-(k+1)^2/n)$. By our assumption that
$n\geq 4k^2/\delta$ (and that $k\geq 2$), this is at least 
$(\delta/9)(k+1)^{-m}$.  
\end{proof}

\section{A modification of an argument of Ajtai and Szemer\'edi} \label{corners}

After Szemer\'edi proved his theorem on arithmetic progressions, it was
natural to try to prove the multidimensional version, which states that 
for every finite subset $H$ of $\mathbb{Z}^d$ and every $\delta>0$ there
exists $N$ such that every subset $A$ of $[N]^d$ of size at least $\delta N^d$
contains a subset of the form $aH+b$ with $a>0$. A full proof of this result
had to wait for the ergodic approach of Furstenberg: the result is due to
Furstenberg and Katznelson \cite{FK78}. However, Ajtai and Szemer\'edi
managed to prove the first genuinely multidimensional case of the theorem, 
where $H$ is the set $\{(0,0),(1,0),(0,1)\}$, by means of a clever deduction
from Szemer\'edi's theorem itself. Their argument is based on a density-increment
strategy, but it is not organized in quite the way that was described in
\S \ref{densityincrement}. However, it is possible to reorganize the
steps so that it follows that general outline very closely: in this section 
we briefly sketch this slight modification of their argument because it 
provides a template for our proof of the density Hales-Jewett theorem.

Let $\delta>0$, let $N$ be a large integer, and let $A$ be a subset of
$[N]^2$ of density at least $\delta$. Our aim is to show that $A$ contains
a triple of the form $\{(x,y),(x+d,y),(x,y+d)\}$ with $d>0$. We shall call
such configurations \textit{corners}. The theorem of Ajtai and Szemer\'edi
is the following.

\begin{theorem} \label{cornerstheorem}
For every $\delta>0$ there exists $N$ such that every subset $A\subset[N]^2$
of density at least $\delta$ contains a triple $\{(x,y),(x+d,y),(x,y+d)\}$ with $d>0$.
\end{theorem}

Before we sketch the proof, we make the general remark that there 
are three privileged directions, horizontal, vertical and parallel to
the line $x+y=0$, which correspond to the three lines that are defined
by pairs of points from the set $\{(0,0),(1,0),(0,1)\}$. Indeed, one could
argue that the formulation of the problem is an unnatural one, and that
instead of the grid $[N]^2$ one should consider a triangular portion of
a triangular lattice, so that there is a symmetry between the three
directions. We shall not do this, but when we come to relate the argument
of this section to the proof of DHJ, it will help to bear this point in mind.

We shall regard certain subsets of $[N]^2$ as ``simple" or ``somewhat 
structured". We define a \textit{1-set} to be a subset of the form $X\times[N]$.
We call such sets 1-sets because whether or not a point $(x,y)$ belongs to
$X\times[N]$ depends only on its first coordinate $x$. A more symmetrical,
and therefore preferable, explanation is this. We represent our points 
not by pairs $(x,y)$ with $x,y\in[N]$ but as \textit{triples} $(x,y,z)$ such
that $x,y\in[N]$ and $x+y+z=2N+1$. (We have chosen $2N$ so that $z$ 
lies between $1$ and $2N-1$, but all we care about is that $x+y+z$ should
be constant.) It is still true that whether or not the point represented by a 
triple $(x,y,z)$ belongs to $X\times[N]$ depends only on $x$. In other
words, if $(x,y,z)$ belongs to a $1$-set, then so does $(x,y+u,z-u)$ for
every $u$. Another way of looking at this, which turns out to correspond
more closely to what we shall do when we prove DHJ, is think of 
a 1-set as a \textit{23-insensitive} set, meaning that membership of
the set is unaffected by changes to the second and third coordinates.

Another special kind of set is one of the form $X\times Y$. This is
the intersection of the 1-set $X\times[N]$ and the 2-set $[N]\times Y$.
In this section we shall call it a 12-set (which is not to be confused
with a 12-insensitive set, which we are calling a 3-set).

Now let us sketch the argument that gives us corners. The basic
idea is a \textit{density increment strategy}, which has been used
to prove many density theorems. (A few examples can be found in
\cite{Rot53}, \cite{Sze75}, \cite{Gow01}, \cite{Shk06a}, \cite{Shk06b},
and \cite{LM08}, but this is by no means an exhaustive list.) We shall
show that if $A$ does not contain a corner, then there is some subset
of $[N]^2$ that looks like $[m]^2$, and inside that subset $A$ has
an increased density. We can iterate this argument until eventually
we reach a contradiction when the relative density of $A$ inside some
subset becomes greater than 1. 

\subsection{Finding a dense diagonal}

The first step is to find a set of the form $\{(x,y):x+y=t\}$ that contains
a reasonable number of points of $A$. Since there are $2N-1$ such 
sets and $A$ has size at least $\delta N^2$, at least one such set
contains at least $\delta N/2$ points of $A$. 

\subsection{A dense 12-set that is disjoint from $A$}

Suppose that we have found $t$ such that the number of points of
$A$ in the diagonal $\{(x,y):x+y=t\}$ is at least $\delta N/2$. Let us
write these points as $(x_1,y_1),\dots,(x_{2m},y_{2m})$ with 
$x_1<\dots<x_{2m}$. If the number of points of $A$ on the diagonal
is odd, we just omit one of them. Let $X=\{x_1,\dots,x_m\}$ and
let $Y=\{y_{m+1},\dots,y_{2m}\}$. Then no point of $X\times Y$ 
can belong to $A$, since if $(x_i,y_j)\in A$ then the three points
$(x_i,y_j),(x_j,y_j)$ and $(x_i,y_i)$ all belong to $A$, and they 
form a corner since $x_j-x_i=y_i-y_j>0$. The size of $X\times Y$
is $m^2$, and $m\geq\lfloor\delta N/4\rfloor$, so (ignoring the
integer part) $X\times Y$ has density at least $\delta^2/16$ or so.

\subsection{A dense 12-set that correlates with $A$}

If $A$ is disjoint from a dense 12-set $X\times Y$ then it must make 
up for this with an increased density in the complement of $X\times Y$.
However, the complement of $X\times Y$ splits up into the three 12-sets
$X\times Y^c$, $X^c\times Y$ and $X^c\times Y^c$. A simple averaging
argument shows that in at least one of these three 12-sets the relative
density of $A$ is at least $\delta+\delta^3/48$. Thus, we have sets
$U$ and $V$ such that the density of $A$ inside the 12-set $U\times V$ is
at least $\delta+\delta^3/48$. Moreover, a very crude argument shows
that the $U\times V$ must have density at least $\delta^3/48$ inside 
$[N]^2$.

\subsection{A dense 1-set can be almost entirely partitioned into large grids}

As mentioned earlier, our eventual aim is to find a subset of $[N]^2$ of a 
similar type, inside which $A$ has increased density. The subsets that
will interest us are \textit{grids}, which are sets of the form $P\times Q$,
where $P$ is an arithmetic progression and $Q$ is a translate of $P$. 

Given a dense 1-set $X\times[N]$, we can partition almost all of it into grids as
follows. Suppose that the density of $X$ is $\theta$ and let $\epsilon$ be some
positive constant that is much smaller than $\theta$ (but independent of $N$).
Since $X$ has density at least $\epsilon$, by Szemer\'edi's theorem it contains
an arithmetic progression $P_1$ of length at least $m$, where $m$ tends to infinity 
with $N$. If the set $X\setminus P_1$ still has density at least $\epsilon$, then it contains 
an arithmetic progression of length $m$. Indeed, we can partition $X$ into 
sets $P_0,P_1,\dots,P_r$, where $P_1,\dots,P_r$ are arithmetic progressions
of length at least $m$ and $P_0$ is a residual set of density less than $\epsilon$. 

For each $i$, we can then straightforwardly partition almost all of $P_i\times[N]$ 
into sets of the form $P_i\times Q_{ij}$, where each $Q_{ij}$ is a translate of
$P_i$. (It helps if each $P_i$ has diameter at most $\epsilon N$, but
it is easy to ensure that this is the case.) We can therefore partition all but an
arbitrarily small proportion of $X\times[N]$ into grids of size tending to infinity
with $N$.

\subsection{A dense 12-set can be almost entirely partitioned into large grids}

It is easy to deduce from the previous step a similar statement about 12-sets.
Indeed, let $X$ and $Y$ be dense sets, and begin by partitioning almost all of
$X\times[N]$ into large grids $P_i\times Q_i$. (We have changed the indexing 
of these grids.) The intersection of $X\times Y$ with any of these grids $P_i\times Q_i$
is $P_i\times(Y\cap Q_i)$, since $P_i\subset X$. Therefore,
if $Y\cap Q_i$ has positive density inside $Q_i$, we can use the previous step
to partition almost all of $P_i\times (Y\cap Q_i)$ into subgrids, still with size
tending to infinity. By a simple averaging argument, the proportion of points 
in $X\times Y$ that are contained in grids $P_i\times Q_i$ inside which $Y$ is sparse
is small. So by this means we have partitioned almost all of $X\times Y$ into
grids with sizes that tend to infinity. 

\subsection{A density increment on a large grid}

By Step 3, we have a dense 12-set $X\times Y$ inside which the density of $A$
is at least $\delta+\delta^3/48$. By Step 5 we can partition almost all of $X\times Y$
into large grids. If we choose ``almost" appropriately, we can ensure that the
density of that part of $A$ that lies in these large grids is at least $\delta+\delta^3/100$.
But then by averaging we can find a large grid $P\times Q$ such that the density
of $A$ inside $P\times Q$ is at least $\delta+\delta^3/100$. This is exactly what
we need for our density-increment strategy, so the proof is complete.

\section{A detailed sketch of a proof of DHJ$_3$}

In this section, we shall explain in some detail how our proof works in the case 
$k=3$. As mentioned in the previous section, the structure of our proof is closely 
modelled on the structure of the argument of Ajtai and Szemer\'edi (in 
the slightly modified form in which we have presented it). However, to make
that clear, we need to explain what the counterparts are of concepts such 
as ``grid", ``12-set" and the like. So let us begin by discussing a dictionary
that will guide us in our proof.

Everything flows from the following simple thought: whereas a typical point
in $[N]^2$ can be thought of as a triple $(x,y,z)$ such that $x+y+z=2N+1$,
a typical point in $[3]^n$ can be thought of as a triple of disjoint sets
$(X,Y,Z)$ such that $X\cup Y\cup Z=[n]$: to turn such a triple into a sequence
$(x_1,\dots,x_n)$ let $x_i=1$ if $i\in X$, 2 if $i\in Y$ and 3 if $i\in Z$.

A corner in $[N]^2$ can be defined symmetrically as a triple of points of the 
form $\{(x+u,y,z),(x,y+u,z),(x,y,z+u)\}$ such that $x+y+z+u=2N+1$ and $u\ne 0$. 
This translates very nicely: a combinatorial line is a triple of points of the form 
$\{(X\cup U,Y,Z),(X,Y\cup U,Z),(X,Y,Z\cup U)\}$ such that $X,Y,Z$ and $U$
partition $[n]$ and $U\ne\emptyset$.

A diagonal in $[N]^2$ is a set of the form $D_t=\{(x,y,z):x+y=t\}$. It therefore
makes sense to define a ``diagonal" in $[3]^n$ to be a set of the form
$\{(X,Y,Z):X\cup Y=T\}$ for some subset $T\subset[n]$. In other words,
it is the collection of all triples $(X,Y,Z)$ that partition $[n]$, but now 
$Z$ is a fixed set (equal to the complement of $T$ above).

Recall that a 1-set in $[N]^2$ is a set of the form $X\times[N]$, or in symmetric
notation a set of the form $\{(x,y,z):x\in X\}$. The obvious generalization of 
this notion to $[3]^n$ is a set of the form $\{(X,Y,Z):X\in\mathcal{X}\}$ for
some collection $\mathcal{X}$ of subsets of $[n]$. A subset $S$ of $[3]^n$
is a 1-set if and only if it is \textit{23-insensitive} in the following sense:
if $(X,Y,Z)\in S$, then $(X,Y',Z')\in S$ whenever $Y'\cup Z'=Y\cup Z$.
Equivalently, if a sequence $x\in [3]^n$ belongs to $S$, then so do all 
sequences that can be formed from $x$ by changing some $2$s to $3$s 
and/or some $3$s to $2$s.

The natural definition of a 12-set is now clear: as in the case of subsets of
$[N]^2$, it should be the intersection of a 1-set with a 2-set.

We should also mention that the notion of Cartesian product has an 
analogue. The Cartesian product of $X$ and $Y$ is the intersection of
the 1-set $X\times[N]$ with the 2-set $[N]\times Y$. So if we are given
two collections $\mathcal{X}$ and $\mathcal{Y}$ of subsets of $[n]$,
then the analogue of their Cartesian product ought to be the 12-set
$\{(X,Y,Z):X\in\mathcal{X},Y\in\mathcal{Y},X\cap Y=\emptyset\}$. 
Since $X$ and $Y$ determine $Z$, we can think of this as a set of pairs, and then the
resemblance with a true Cartesian product is that much closer: it
is (equivalent to) the set of all pairs $(X,Y)$ such that $X\in\mathcal{X}$,
$Y\in\mathcal{Y}$ \textit{and $X$ and $Y$ are disjoint}. We
shall call this the \textit{disjoint product} of $\mathcal{X}$ and $\mathcal{Y}$
and write it as $\mathcal{X}\boxtimes\mathcal{Y}$.

There is one concept that has a non-obvious (though still natural) 
translation from the $[N]^2$ world to the $[3]^n$ world, namely that
of a grid. At first sight, it might seem extremely unlikely that the 
Ajtai-Szemer\'edi can be generalized to give a proof of DHJ$_3$. After all, their proof 
could be regarded as the beginnings of a sort of induction: they deduce
the first non-trivial case of the two-dimensional theorem from the 
full one-dimensional theorem (namely Szemer\'edi's theorem).
If one is attempting to prove DHJ$_3$, the obvious candidate for 
a statement ``one level down" is DHJ$_2$, but that is a much less
deep statement than Szemer\'edi's theorem. So it seems that our only
hope will be if Ajtai and Szemer\'edi did not after all need a tool as
powerful as Szemer\'edi's theorem.

One of the key ideas of our proof is that this is indeed the case, though
the result we need is not DHJ$_2$ but its multidimensional version
MDHJ$_2$ proved in the last section. The appropriate replacement
of the notion of a long arithmetic progression in $[N]$ is a combinatorial
subspace of $[2]^n$. We then have to decide what the analogue of
a grid is. Given the concepts so far, it should be something like the
disjoint product of two ``parallel" combinatorial subspaces of
$[2]^n$, and we would like that to give us a combinatorial subspace
of $[3]^n$ (since we want the analogue of a grid to be a structure 
that resembles $[3]^n$). All this can be done. A $d$-dimensional
combinatorial subspace of $[2]^n$ is defined by taking disjoint sets 
$X_0,X_1,\dots,X_d$ and defining $U$ to be the set of all unions
$X_0\cup\bigcup_{i\in A}X_i$ such that $A\subset[d]$. It is natural
to define two such subspaces to be parallel if they are defined by
sequences of sets $(X_0,X_1,\dots,X_d)$ and $(Y_0,Y_1,\dots,Y_d)$
such that $X_i=Y_i$ for every $i\geq 1$, and also, since we want to 
take a disjoint product, to add the condition that $X_0$ and $Y_0$ 
should be disjoint. If we do that, then a typical point in the 
disjoint product is a pair $(X,Y)$ such that $X=X_0\cup\bigcup_{i\in A}X_i$
and $Y=Y_0\cup\bigcup_{i\in B}X_i$ such that $A\cap B=\emptyset$.
If we set $Z=[n]\setminus(X\cup Y)$, we see easily that this is
precisely a $d$-dimensional combinatorial subspace of $[3]^n$: $X_0$ and $Y_0$
are the sets where the fixed coordinates are 1 and 2, respectively,
and the wildcard sets are $X_1,\dots,X_d$.

With these concepts in mind, let us now give an overview of the
proof of DHJ$_3$. (To generalize this discussion to DHJ$_k$ is
straightforward: the Ajtai-Szemer\'edi argument can be used to
deduce a ``$k$-dimensional corners" theorem from the
$(k-1)$-dimensional Szemer\'edi theorem, and that provides a 
template for our deduction of DHJ$_{k+1}$ from MDHJ$_k$,
which itself can be deduced from PDHJ$_k$, which follows
from DHJ$_k$.)

\subsection{Finding a dense diagonal}

Recall that we are defining a diagonal in $[3]^n$ to be a set
of the form $\{(X,Y,Z):X\cup Y=T\}$. Equivalently, one fixes a
set $Z$ and defines the associated diagonal to be the set of
all sequences in $[3]^n$ that take the value $3$ in $Z$ and
1 or 2 everywhere else. 

Obviously the diagonals form a partition of $[3]^n$, so if 
$A\subset[3]^n$ is a set of density $\delta>0$ then by averaging 
we can find a diagonal inside which $A$ still has density $\delta$. 
We can also ensure that this diagonal is not too small by throwing
away the very small fraction of $[3]^n$ that is contained in small
diagonals.

It is not completely obvious at this stage what probability measure 
we want to take on $[3]^n$, but note that the argument so far is 
general enough to apply to any measure.

\subsection{A dense 12-set that is disjoint from $A$}

What should we do next? In the equivalent stage of the corners
argument we were assuming that $A$ contained no corners. Then 
every pair of points of $A$ in our dense diagonal implied that a
third point (the bottom of the corner of which those two points
formed the diagonal) did not belong to $A$. Moreover, the set
of points that we showed did not belong to $A$ formed a 
dense 12-set. So now we would like to do something similar.

At first, the situation looks very promising, since if $(X,Y,Z)$
and $(X',Y',Z)$ are two points with $X\subset X'$, both belonging
to the diagonal determined by the set $Z$,
then we can set $U=X'\setminus X$ and write these two points 
as $(X,Y\cup U,Z)$ and $(X\cup U,Y,Z)$. Then the point
$(X,Y,Z\cup U)$ cannot lie in $A$, since otherwise the three
points would form a combinatorial line in $A$. 

So what can we say about the set of all forbidden points? These
are all points of the form $(X,Y,Z\cup U)$ such that both
$(X\cup U,Y,Z)$ and $(X,Y\cup U,Z)$ belong to $A$. Now
$Z$ is a fixed set (that defines the particular diagonal we are
talking about), so if we are presented with a point $(X,Y,Z\cup U)$
then we can work out what $U$ is. Let $\mathcal{X}$ be the 
set of all $X\subset [n]\setminus Z$ such that 
$(X,[n]\setminus(X\cup Z),Z)\in A$. Then the set of all
$(X,Y,Z\cup U)$ such that $(X,Y\cup U,Z)\in A$ is precisely
the set of all $(X,Y,Z\cup U)$ such that $X\in\mathcal{X}$.
This would be a 1-set if we were not insisting that every
point took the value 3 in the set $Z$. However, the set of
all such points forms a subspace of $[3]^n$ (of dimension
$n-|Z|$), and inside that subspace we have a 1-set. 
Similarly, the set of all $(X,Y,Z\cup U)$ such that 
$(X\cup U,Y,Z)\in A$ is a 2-set inside the same subspace:
this time we define $\mathcal{Y}$ to be the set of all 
$Y$ such that $([n]\setminus(Y\cup Z),Y,Z)\in A$ and 
take the set of all points $(X,Y,Z\cup U)$ such that
$Y\in\mathcal{Y}$. 

Thus, the good news is that we have found a 12-set
that is disjoint from $A$, but the bad news is that this
12-set is in a subspace of $[3]^n$ rather than in the
whole space. 

\subsection{A dense 12-set that correlates with $A$}

In the proof of the result about corners, we used a simple
averaging argument at this stage: if there is a dense 
12-set that is disjoint from $A$ then one of three other
12-sets must have an unexpectedly large intersection
with $A$. However, we cannot argue as straightforwardly
here, since the 12-set we have found is \textit{not} dense.

There are in fact two problems here. The first is the obvious
one that we have restricted to a subspace, the density of
which will be very small. To see this, note that for almost
all points $(X,Y,Z)$ in $[3]^n$ the sets $X$, $Y$ and $Z$
have size very close to $n/3$. Therefore, it may well be
that $A$ consists solely of such points, in which
case when we pass to the subspace that takes the value
3 on some fixed $Z$ we will lose approximately $n/3$
dimensions. 

The second problem is that even when we do restrict
to such a subspace we find that $A$ may well have
tiny density, since almost all triples in such a subspace
will be of the form $(X,Y,Z\cup U)$ with $X$, $Y$ and
$U$ all of approximately the same size, and it may
well be that no such triples belong to $A$, since then
$X$, $Y$ and $Z\cup U$ do \textit{not} all have approximately
the same size. 

To get round these problems, we do two things. First,
we do not use the uniform measure on $[3]^n$ but instead
the equal-slices measure. This deals with the second 
problem, since for an equal-slices random triple $(X,Y,Z)$ 
it is no longer the case that the sets $X$, $Y$ and 
$Z$ almost always have approximately the same
size. Second, we argue that we may assume that the 
restriction of $A$ to almost all subspaces has density
at least $\delta-\eta$ for some very small $\eta$. This 
observation is standard in proofs of density theorems: 
roughly speaking, if $A$ often has smaller density than
this, then somewhere it must have substantially larger 
density (by averaging), and then we have completed
the iteration step in a particularly simple way. But if
$A$ almost always has density at least $\delta-\eta$,
then when we use an averaging argument to find
a diagonal that contains many points of $A$, we
can also ask for $A$ to have density at least $\delta-\eta$
inside the subspace we are forced to drop down to.

Once all these arguments have been made precise,
the conclusion is that there is a subspace $V$ of $[3]^n$
of reasonably large dimension such that the density of
$A$ inside $V$ is at least $\delta-\eta$, and a dense 12-set 
inside that subspace that is disjoint from $A$. Then a 
simple averaging argument similar to the one in the
corners proof gives us a dense 12-set in that subspace
inside which the relative density of $A$ is at least $\delta+c(\delta)$.
(For this we must make sure we choose $\eta$ sufficiently 
small for the small density decrease to be more than
compensated for by the subsequent density increase.)

Thus, although the statement and proof of this step are
directly modelled on the corresponding step for the corners
proof, there are some important differences: we show that $A$
correlates \textit{locally} (that is, in some subspace of 
density that tends to zero) with a 12-set, whereas in the
corners proof a global correlation is found. We do not know
whether a dense subset of $[3]^n$ that contains no 
combinatorial line must correlate globally with a 12-set.
(Strictly speaking, we \textit{do} know, since we have
proved that every dense subset of $[3]^n$ contains a
combinatorial line. However, one can obtain a better
formulation of the question by replacing the assumption
that the set contains no lines by the assumption that it
contains few lines.) A second difference is that although
we start with a set $A$ that is equal-slices dense, the
local correlation that the proof ends up giving is with
respect to the uniform measure. (There is a general
principle operating here, which is that equal-slices
measure does not behave well when you restrict
to combinatorial subspaces.)

\subsection{A dense 1-set can be almost entirely partitioned into large combinatorial subspaces}
\label{partitioning}

Bearing in mind our dictionary, the next stage of the proof should be
to partition almost all of a dense 1-set into combinatorial subspaces
of dimension tending to infinity.

Let us recall what a 1-set, or a 23-insensitive set, is. It is a set $A\subset[3]^n$
with the property that if $x\in A$, $y\in[3]^n$ and $\{i:x_i=1\}=\{i:y_i=1\}$, then $y\in A$.
Equivalently, using set-theoretic notation, it is a set of triples of the
form $\{(X,Y,Z):X\in\mathcal{X}\}$ for some collection $\mathcal{X}$ of
subsets of $[n]$.

At this stage of the corners proof, one starts with a 1-set $X\times [N]$,
applies Szemer\'edi's theorem over and over again to remove arithmetic 
progressions $P_i$ from $X$ until it is no longer dense, and then partitions
the sets $P_i\times[N]$ into sets of the form $P_i\times Q_{ij}$, where the
$Q_{ij}$ are translates of $P_i$. 

If we follow the proof of the corners theorem, then we should expect
an argument along the following lines. We start with the 1-set
$\{(X,Y,Z):X\in\mathcal{X}\}$. We then partition almost all of $\mathcal{X}$, 
which can be thought of as a subset of $[2]^n$, into large combinatorial 
subspaces using repeated applications of MDHJ$_2$. For each one
of these subspaces $U$, we then partition the disjoint product 
$U\boxtimes[3]^n$ into combinatorial subspaces.

Unfortunately, this last step does not work, which leads us to the second
point where our argument is more complicated than that of Ajtai and Szemer\'edi,
and the second place where we use localization to get us out of trouble.
The difficulty is this. If $U$ is the $d$-dimensional subspace defined by 
the sets $(X_0,X_1,\dots,X_d)$, then $U\boxtimes[3]^n$ consists of 
all triples $(X,Y,Z)$ of disjoint sets such that $X$ is a union of $X_0$
with some of the sets $X_i$. A combinatorial subspace inside this set must 
have wildcard sets that are unions of the $X_i$ with $i\geq 1$, which means 
that it cannot contain any point $(X,Y,Z)$ such that $Y\cap X_i$ and 
$Z\cap X_i$ are non-empty for every $i$.  

This is a genuine difficulty, but we can get round it. The way we do so may at
first look a little dangerous, but it turns out to work. The argument proceeds
in five steps as follows. 

\begin{itemize}
\item Let $B$ be a 23-insensitive set of density $\eta$. Let $m$ be a positive
integer to be chosen later (for now it is sufficient to think of it as a number that
tends to infinity but is much much smaller than $n$), and choose a random 
element of $[3]^n$ by randomly
permuting the ground set $[n]$ and then taking a pair $(x,y)$, where $x$ is chosen 
uniformly from $[2]^m$ and $y$ is chosen uniformly from $[3]^{n-m}$. (Here we
are regarding $x$ as supported on the first $m$ elements of the randomly 
permuted ground set and $y$ as supported on the last $n-m$ elements.)
For sufficiently small $m$, the distribution of $(x,y)$ is approximately uniform,
so if for each $y$ we let $E_y=\{x:(x,y)\in B\}$, then $E_y$ has density at
least $\eta/3$ in $[2]^m$ for a set of $y$ of density at least $\eta/3$. (This is
not the main reason that we need $m$ to be small, so this step will be
true with a great deal of room to spare.)

\item For each such $y$ use MDHJ$_2$ to find a $d$-dimensional 
combinatorial subspace $U$ of $[2]^m$ that lives inside $E_y$, and
hence has the property that $(x,y)\in B$ for every $x\in U$. (Here, $d$ 
depends on $m$ and $\eta$.)

\item By the pigeonhole principle, we can find a subset $T$ of $[3]^{n-m}$ 
of density $\theta=\theta(m,d,\eta)$ and a combinatorial subspace
$U\subset[2]^m$ such that $U\times T\subset B$. Let us choose $T$
to be maximal: that is, $T$ is the set of all $y\in[3]^{n-m}$ such that
$U\times\{y\}\subset B$. Since $B$ is a 
23-insensitive set, it follows that if we allow the wildcard sets of
$U$ to take the value 3 as well, then all the resulting points will still belong
to $B$. That is, we have the same statement as above but now $U$
is a combinatorial subspace of $[3]^m$. This is the point of our
argument ``where the induction happens".

\item $U\times T$ is a union of combinatorial subspaces, and there are
quite a lot of them. It is tempting at this stage to remove them from $B$
and start again. But unfortunately there is no reason to suppose that
$B\setminus (U\times T)$ will be 23-insensitive. (We give an example
to illustrate this just after this proof outline.)
However, this turns out not to be too serious a problem, because for
every $x\in X$ the set $(B\setminus(U\times T))\cap(\{x\}\times[3]^{n-m})$
is a 23-insensitive subset of $\{x\}\times[3]^{n-m}$. In other words, we
can partition $B\setminus(U\times T)$ into locally 23-insensitive sets
and run the argument again.

\item Using this basic idea, we develop an iterative proof. Whenever
we are faced with a set of small density we regard it as part of our
``error set" and leave it alone. And from any set of large density we
remove a disjoint union of combinatorial subspaces and partition
the rest into locally 23-insensitive sets. If we are careful, we can
choose $m$ in such a way that the combinatorial subspaces have
dimension that tends to infinity with $n$, but the number of iterations
before there are no dense sets left is smaller than $n/m$, so we never 
``run out of dimensions". In this way we prove that a 23-insensitive
set can almost all be partitioned into combinatorial subspaces.
\end{itemize}

Here, as promised, is an example of a 23-insensitive set $B$ such
that removing $U\times T$ leaves us with a set that is no longer
23-insensitive. Let $m=2$ and $n=3$ and let $B$ be the 23-insensitive set 
$\{11, 22, 23, 32, 33\}\times\{2,3\}$. Then $B$ contains the set
$\{11,22,33\}\times\{1,2,3\}$, which is of the form $U\times T$ with
$U$ a subspace and $T$ 23-insensitive (and it is the only non-empty
subset of this form). If we remove this from $B$, we end up with the 
set $\{23,32\}\times\{2,3\}$, which is no longer 23-insensitive. It
is, however, 23-insensitive in the third coordinate.

\subsection{A dense 12-set can be almost entirely partitioned into large
combinatorial subspaces}

This stage of the argument is very similar to the corresponding stage
of the corners argument and needs little comment. One simply checks
that the intersection of a 13-insensitive set with a combinatorial subspace
is 13-insensitive inside that subspace (which is almost trivial). Then,
given an intersection of a 23-insensitive set and a 13-insensitive set,
one applies the result of the previous section to the 23-insensitive
set, partitioning almost all of it into subspaces, and then applies the
same argument to the 13-insensitive set inside each subspace.

\subsection{A density increment on a large combinatorial subspace}

Again, this stage of the argument is very similar to the corresponding 
stage of the corners argument. If $A$ has increased density on a 
(locally) 23-insensitive set, and if that set can be almost entirely
partitioned into combinatorial subspaces of dimension tending to
infinity, then by averaging we must be able to find one of these
combinatorial subspaces inside which $A$ has increased density.

We are not quite in a position to iterate at this point, because we 
started out with a set of equal-slices measure $\delta$ and ended 
up finding a combinatorial subspace on which the \textit{uniform}
density had gone up. However, it turns out to be quite easy to 
pass from that to a further subspace inside which $A$ has 
an equal-slices density increment, at which point we are done.

\section{Measure for measure}

As we have already mentioned, there are some arguments that
work better when we use product measures, and others when
we use equal-slices measures. This appears to be an unavoidable
situation, so we need a few results that will tell us that if we can
prove a statement in terms of one measure then we can deduce
a statement in terms of another. In this section, we shall collect
together a number of such results, so that later on in the paper
we can simply apply them when the need arises. The results we prove
are just technical calculations, so the reader may prefer to take 
them on trust. The statements we shall need later are Corollary
\ref{uniformtoequalslices}, Corollary \ref{uniformtosmallerequalslices}
and Lemma \ref{equalslicestouniform}.

We begin with a standard definition that will tell us when we regard 
two probability measures as being close.

\begin{definition}
Let $\mu$ and $\nu$ be two probability measures on a finite set $X$.
The \emph{total variation distance} $d(\mu,\nu)$ is defined to be
$\max_{A\subset X}|\mu(A)-\nu(A)|$. 
\end{definition}

In order to prove that we can switch from one probability measure to
another, we shall make use of the following very simple general
principle.

\begin{lemma}\label{measuretransfer}
Let $\mu$ and $\nu_1,\dots,\nu_m$ be probability measures, let
$a_1,\dots,a_m$ be positive real numbers that add up to 1, and
suppose that $d(\mu,\sum_{i=1}^ma_i\nu_i)\leq\eta$. Then for
every $\alpha\in[0,1]$ and every set $A$ such that $\mu(A)\geq\alpha$
there exists $i$ such that $\nu_i(A)\geq\alpha-\eta$.
\end{lemma}

\begin{proof}
From our assumptions it follows that $\sum_{i=1}^ma_i\nu_i(A)\geq\alpha-\eta$,
so by averaging it follows that there exists $i$ such that $\nu_i(A)\geq\alpha-\eta$.
\end{proof}

\subsection{From uniform measure to equal-slices measure}

Before we apply Lemma \ref{measuretransfer}, let us prove a simple but 
useful technical lemma.

\begin{lemma}\label{balanced}
Let $x$ be an element of $[k]^n$ chosen uniformly at random, and for each
$j\in[k]$ let $X_j=\{i:x_i=j\}$. Then with probability at least $1-2k\exp(-2n^{1/3})$ 
the sets $X_j$ all have size between $n/k-n^{2/3}$ and $n/k+n^{2/3}$.
\end{lemma}

\begin{proof}
The size of $X_j$ is binomial with parameters $n$ and $1/k$. Standard
bounds for the tail of the binomial distribution therefore tell us that the
probability that $|X_i|$ differs from $n/k$ by at least $r$ is at most
$2\exp(-2r^2/n)$. (This particular bound follows from Azuma's inequality.)
The result follows.
\end{proof}

As a first application of Lemma \ref{measuretransfer} we shall prove
that a set of uniform density $\delta$ has equal-slices density almost
as great on some combinatorial subspace. The actual result we shall
prove is, however, slightly more general. To set it up, we shall need a
little notation. 

Let $m<n$, let $\sigma$ be an injection from $[m]$ to $[n]$, let $J=\sigma([m])$ 
and let $\overline{J}$ be the complement of $J$. Then we can write each 
element of $[k]^n$ as a pair $(x,y)$ with $x\in[k]^J$ and $y\in [k]^{\overline{J}}$. 
An element of $[k]^J$ is a function from $J$ to $[k]$. Given an element
$x=(x_1,\dots,x_m)$ of $[k]^m$, let $\phi_\sigma(x)$ be the element 
of $[k]^J$ that takes $j\in J$ to $x_{\sigma^{-1}(j)}$. In other words,
$\phi_\sigma$ takes an element of $[k]^m$ and uses $\sigma$ to turn
it into an element of $[k]^J$ in the obvious way. Given $y\in[k]^{\overline{J}}$,
we also define a map $\phi_{\sigma,y}:[k]^m\rightarrow[k]^n$ by taking
$\phi_{\sigma,y}(x)$ to be $(\phi_\sigma(x),y)$. Thus, $\phi_{\sigma,y}$
is a bijection between $[k]^m$ and the combinatorial subspace 
$S_{J,y}=\{(x,y):x\in [k]^J\}$ (in which the wildcard sets are all
singletons $\{i\}$ such that $i\in J$). 

Now let $\nu$ be a probability measure on $[k]^m$. For each pair
$(\sigma,y)$ as above, we can define a probability measure $\nu_{\sigma,y}$
on $[k]^n$ by ``copying" $\nu$ in the obvious way. That is, given a subset 
$A\subset[k]^n$ we let $\nu_{\sigma,y}(A)=\nu(\phi_{\sigma,y}^{-1}(A))$. 

We now show that if $m$ is sufficiently small, then the average of all
the measures $\nu_{\sigma,y}$ is close to the uniform measure on $[k]^n$.

\begin{lemma} \label{averagemeasure}
Let $\eta>0$, let $k\geq 2$ be a positive integer, let $n\geq(16k/\eta)^{12}$,
let $m\leq n^{1/4}$, let $\nu$ be a probability measure
on $[k]^m$ and let $\mu$ be the uniform measure on $[k]^n$. Then 
$d(\mu,\mathbb{E}_{\sigma,y}\nu_{\sigma,y})\leq\eta$, where the
average is over all pairs $(\sigma,y)$ as defined above.
\end{lemma}

\begin{proof}
We shall prove the result in the case where all of $\nu$ is concentrated
at a single point. Since all other probability measures are convex 
combinations of these ``delta measures" (and their copies are the same
convex combinations of the copies of the delta measures), the result will 
follow.

Let $u$, then, be an element of $[k]^m$ and for each $C\subset[k]^m$ let 
$\nu(C)=1$ if $u\in C$ and $0$ otherwise. For each injection $\sigma:[m]\rightarrow[n]$
and each $y\in[k]^{\overline{J}}$ (where $\overline{J}$ is again the complement
of $\sigma([m])$), the measure $\nu_{\sigma,y}$ is the delta measure at 
$\phi_{\sigma,y}(u)$. That is, $\nu_{\sigma,y}(A)=1$ if $\phi_{\sigma,y}(u)\in A$ 
and $\nu_{\sigma,y}(A)=0$ otherwise. 

What, then, is $\mathbb{E}_{\sigma,y}\nu_{\sigma,y}(A)$? To answer this, let us 
see what happens when $A$ is a singleton $\{z\}$. Then $\nu_{\sigma,y}(A)=1$
if and only if the restriction of $z$ to $J$ is $\phi_\sigma(u)$ and the restriction
of $z$ to $\overline{J}$ is $y$. So $\mathbb{E}_{\sigma,y}\nu_{\sigma,y}(A)$ is the probability,
for a randomly chosen pair $(\sigma,y)$, that $z_{\sigma(i)}=u_i$ for every 
$i\in[m]$ and the restriction of $z$ to $\overline{J}$ is $y$. 

For every $\sigma$, the probability of the second event given $\sigma$ is 
$k^{m-n}$, so it remains to calculate the probability that $z_{\sigma(i)}=u_i$
for every $i$. For each $j\in[k]$, let $X_j=\{i:z_i=j\}$ and let $n_j$ be the 
cardinality of $X_j$. Now let us choose the values $\sigma(1),\sigma(2),\dots,
\sigma(m)$ one at a time and estimate the conditional probability that
$\sigma(i)\in X_{u_i}$ given that $\sigma(h)\in X_{u_h}$ for every $h<i$.
If we set $p=\min_jn_j$ and $q=\max_jn_j$, then each conditional 
probability of this kind will be at most $q/(n-m)$ and at least $(p-m)/(n-m)$.

Lemma \ref{balanced} tells us that with probability at least $1-2k\exp(-2n^{1/3})$
we have the bounds $n/k-n^{2/3}\leq p$ and $q\leq n/k+n^{2/3}$. If those
bounds hold, then the probability that $\sigma(i)\in X_{u_i}$ for every
$i\in[m]$ lies between $(1/k-2n^{-1/3})^m$ and $(1/k+2n^{-1/3})^m$.
(Here we are using the inequality that $(n/k+n^{2/3})/(n-n^{1/4})\leq 1/k+2n^{-1/3}$,
which holds if $k\geq 2$ and $n\geq 8$.) Therefore, it lies between
$k^{-m}(1-\eta/4)$ and $k^{-m}(1+\eta/4)$. (This inequality
is valid if $n\geq (16k/\eta)^{12}$, as we are assuming.)

We have just shown that the value of the measure $\mathbb{E}_{\sigma,y}\nu_{\sigma,y}$
on a singleton $\{z\}$ is approximately equal to the value taken by the uniform
measure, provided that the singleton has roughly the same number of coordinates
of each value. 

Let $B$ be the set of all ``balanced" sequences $z$. That is, $B$ is the set
of $z$ such that the assumptions of the above argument are satisfied. Then 
$\mathbb{E}_{\sigma,y}\nu_{\sigma,y}(B)\geq(1-2k\exp(-2n^{1/3})(1-\eta/4)\geq 1-\eta/2$,
from which it follows that $\mathbb{E}_{\sigma,y}\nu_{\sigma,y}(B^c)\leq\eta/2$.
Therefore, if $A$ is any subset of $[k]^n$, we have that
\begin{equation*}
\mathbb{E}_{\sigma,y}\nu_{\sigma,y}(A)\leq\mu(A)(1+\eta/4)+\eta/2
\end{equation*}
and 
\begin{equation*}
\mathbb{E}_{\sigma,y}\nu_{\sigma,y}(A)\geq\mu(A)(1-\eta/4)-\eta/2.
\end{equation*}
Since $\mu(A)\leq 1$, it follows that $|\mu(A)-\mathbb{E}_{\sigma,y}\nu_{\sigma,y}(A)|\leq\eta$.

As commented at the beginning of the proof, the result for arbitrary $\nu$ follows
from this result, since we can write it as a convex combination of delta measures
and apply the triangle inequality.
\end{proof}

Armed with this result, we now prove two statements that will be helpful to us
later on. 

\begin{corollary} \label{uniformtoequalslices}
Let $A$ be a subset of $[k]^n$ of uniform density $\delta$, let $\eta>0$, let $m\leq n^{1/4}$ 
and suppose that $n\geq(16k/\eta)^{12}$. Let $J$ be a random subset of $[n]$ of size
$m$ and let $y$ be a random element of $[k]^{\overline{J}}$. Then the expected equal-slices
density of $A$ inside the combinatorial subspace $S_{J,y}$ is at least $\delta-\eta$. In particular,
there exist $J$ and $y$ such that the equal-slices density of $A$ inside $S_{J,y}$ 
is at least $\delta-\eta$. 
\end{corollary}

\begin{proof}
Let $\nu$ be the equal-slices measure on $[k]^m$ and apply Lemma \ref{averagemeasure}.
It implies that $\mathbb{E}_{\sigma,y}\nu_{\sigma,y}(A)\geq\delta-\eta$, from which
it follows that there exists a pair $(\sigma,y)$ such that $\nu_{\sigma,y}(A)\geq\delta-\eta$.
But $\nu_{\sigma,y}$ is the equal-slices measure on the combinatorial subspace
$S_{J,y}$, where $J=\sigma([m])$, which is $m$-dimensional.
\end{proof}

For the next lemma we need some notation. Given a subset $J\subset [n]$ of size
$m$ and a sequence $y\in[k]^{\overline{J}}$, let us write $S'_{J,y}$ for the 
set of all sequences in $S_{J,y}$ that never take the value $k$ in $J$. Thus, $S'_{J,y}$
is a copy of $[k-1]^m$. By the equal-slices density on $S'_{J,y}$ we mean the
image of the equal-slices density on $[k-1]^m$ (where this is considered as a
set in itself and not as a subset of $[k]^m$). 

\begin{corollary} \label{uniformtosmallerequalslices}
Let $A$ be a subset of $[k]^n$ of uniform density $\delta$, let $\eta>0$, let $m\leq n^{1/4}$ 
and suppose that $n\geq(16k/\eta)^{12}$. Let $J$ be a random subset of $[n]$ of size
$m$ and let $y$ be a random element of $[k]^{\overline{J}}$. Then the expected equal-slices
density of $A$ inside the set $S'_{J,y}$ is at least $\delta-\eta$. In particular,
there exist $J$ and $y$ such that the equal-slices density of $A$ inside $S'_{J,y}$ 
is at least $\delta-\eta$. 
\end{corollary}

\begin{proof}
Let $\nu'$ be the measure on $[k]^m$ defined by taking $\nu'(A)$ to be the 
equal-slices measure of $A\cap[k-1]^m$ (considered as a subset of $[k-1]^m$).
In other words, $\nu'(A)$ is the probability that $x\in A$ if you choose a random
$(k-1)$-tuple $(r_1,\dots,r_{k-1})$ of positive integers that add up to $m$ and
then let $x$ be a random element of $[k-1]^m$ with $r_j$ $j$s for each $j$.

Applying Lemma \ref{averagemeasure}, we find that 
$\mathbb{E}_{\sigma,y}\nu'_{\sigma,y}(A)\geq\delta-\eta$, from which
it follows that there exists a pair $(\sigma,y)$ such that $\nu'_{\sigma,y}(A)\geq\delta-\eta$.
But $\nu'_{\sigma,y}$ is the equal-slices measure on the set $S'_{J,y}$, where $J=\sigma([m])$.
\end{proof}

\subsection{From equal-slices measure to uniform measure}

We would now like to go in the other direction, passing from a set of 
equal-slices density $\delta$ to a subspace inside which the uniform density 
is at least $\delta-\eta$ for some small $\eta$. As before, we need to use
a result that says that a typical sequence $x$ is not too imbalanced. Since 
we are choosing $x$ from the equal-slices measure, the conclusion we can
hope for is much weaker than the conclusion of Lemma \ref{balanced}:
the result we use is Lemma \ref{nottooimbalanced}, which tells us
that with high probability every value will be taken a reasonable number
of times.

The result we prove in this subsection states that if $A$ has equal-slices
density $\delta$, then there is
a distribution on the $m$-dimensional subspaces of $[k]^n$ such
that if you choose one at random then the expected uniform density of
$A$ in that subspace is at least $\delta-\beta$.

\begin{lemma} \label{equalslicestouniform}
Let $\delta,\beta>0$, let $m,n$ and $k$ be positive integers, and 
suppose that $m\leq\min\{\beta n/8k,\beta n/2k^2\}$.
Let $A$ be a subset of $[k]^n$ of equal-slices density $\delta$. Let
$J$ be a random subset of $[n]$ of size $[m]$, let $x$ be chosen 
uniformly at random from $[k]^J$ and let $y$ be chosen randomly, 
according to equal-slices measure, from $[k]^{\overline{J}}$
(with this choice made independently of $x$). Then the probability 
that $(x,y)\in A$ is between $\delta-\beta$ and $\delta+\beta$.
\end{lemma}

\begin{proof}
Let $z$ be an element of $[k]^n$. We shall estimate the probability
that $(x,y)=z$, when $x$ and $y$ are chosen as in the statement 
of the theorem, and compare that with the equal-slices probability
of the singleton $\{z\}$. To do this, let us define $u_j$, for each 
$j\in [k]$, to be the number of $i$ such that $x_i=j$. Let us also
assume that $u_j\geq 1$ for every $j$.

We start by considering the case $m=1$. In other words, we first
pick a random $i$ and randomly choose some $j\in[k]$. Then we
randomly choose $y$ from equal-slices measure on $[k]^{[n]\setminus\{i\}}$.
And then we would like to know the probability that $j=z_i$ and
$y_h=z_h$ for every $h\ne i$. 

The probability that $j=z_i$ is $1/k$, since we chose $j$ uniformly.
Now let us suppose that $z_i$ is in fact equal to 1. Then the 
probability that $y_h=z_h$ for every $h\ne i$ is the equal-slices
measure of a singleton that consists of a sequence in $[k-1]^{n-1}$ 
with $u_1-1$ 1s and $u_j$ $j$s for every $j>1$. That measure is
equal to
\begin{equation*}
\binom{n+k-2}{k-1}^{-1}\binom{n-1}{u_1-1,u_2,\dots,u_k}^{-1}
\end{equation*}
(It is here that we are assuming that $u_1\ne 0$.) For comparison, 
the equal-slices measure of $\{z\}$ in $[k]^n$ is
\begin{equation*}
\binom{n+k-1}{k-1}^{-1}\binom{n}{u_1,u_2,\dots,u_k}^{-1}.
\end{equation*}
The first measure divided by the second equals $(n+k-1)/u_1$.

It follows that the probability that $(x,y)=z$ given that
$z_i=1$ is $(n+k-1)/ku_1$. Therefore, by the law of total
probability, the probability that $(x,y)=z$ is
\begin{equation*}
\sum_{j=1}^k\frac 1k\frac{u_j}n\frac{n+k-1}{u_j}=\frac{n+k-1}n
\end{equation*}
times the equal-slices probability of $z$.

Now let us consider the more general case where $|J|=m$.
Again we shall look at the probability that $(x,y)=z$, but this
time we shall assume that $u_j\geq m$ for every $j$. We claim
that the probability of getting $z$ is 
\begin{equation*}
r_{n,k,m}=\frac{(n+k-1)(n+k-2)\dots(n+k-m)}{n(n-1)\dots(n-m+1)}
\end{equation*}
times the equal-slices measure of $\{z\}$. This follows easily
from what we have just done and induction. Indeed, by induction
we know that if we choose a random set $J'$ of size $m-1$ and
choose $x$ uniformly from $[k]^{J'}$ and $y$ using equal-slices
from $[k]^{\overline{J'}}$, then the probability that $(x,y)=z$ is
$r_{n,k,m-1}$ times the equal-slices measure of $\{z\}$. If we
now change the way we choose $y$ by uniformly picking one
coordinate and using equal-slices to pick the rest, then by the
case $m=1$ we multiply this probability by a further $(n+k-m)/(n-m+1)$,
which gives us $r_{n,k,m}$ times the equal-slices measure of $\{z\}$.

Now $r_{n,k,m}$ is at least 1 and at most $(1+k/(n-m))^m$. Given 
our assumption about $m$, this is at most $1+\beta/2$. Thus, for
every $z$ with at least $m$ coordinates of each value, the 
probability that $(x,y)=z$ lies between $\nu(\{z\})$ and 
$(1+\beta/2)\nu(\{z\})$, where $\nu$ is equal-slices measure.
By Lemma \ref{nottooimbalanced}, the equal-slices probability 
that $z$ does not have at least $m$ coordinates of each value 
is at most $mk^2/n$, which by assumption is at most $\beta/2$.

Now let $A$ be any subset of $[k]^n$ of density $\delta$ and
let $B$ be the set of all sequences such that for some $j$ there
are fewer than $m$ coordinates equal to $j$. Then if we choose
$(x,y)$ randomly in the manner stated, the probability that it
belongs to $A$ is at most $(1+\beta/2)\nu(A)+\beta/2$, since
the probability that it belongs to $B$ is at most the equal-slices
measure of $B$ (as we see by looking at $B^c$). The probability
is also at least $\nu(A)-\beta/2$, for similar reasons.  This proves 
the lemma.
\end{proof}

We now show that DHJ implies the equal-slices version of DHJ
(which we stated earlier as Theorem \ref{equalslicesdhj}).

\begin{corollary} \label{dhjimpliesedhj}
Let $k$ be a positive integer and suppose that DHJ$_k$ is true. 
Let $\delta>0$ and let $n\geq (16k^2/\delta)$DHJ$(k,\delta/2)$. Then
every subset of $[k]^n$ of equal-slices density at least $\delta$
contains a combinatorial line.
\end{corollary}

\begin{proof}
By Lemma \ref{equalslicestouniform} there exists a combinatorial
subspace $V$ of dimension not less than DHJ$(k,\delta/2)$ such that 
the uniform density of $A$ in $V$ is at least $\delta/2$. The result
follows.
\end{proof}

\subsection{From uniform measure on $[k]^n$ to uniform measure on $[k-1]^m$}

We need one more result of a similar kind. This time it says that if we
choose a random set $J\subset[n]$ of size $m$ and choose $y$
uniformly at random from $[k]^{\overline{J}}$ and $x$ uniformly at
random from $[k-1]^J$, then the distribution of $(x,y)$ is approximately
uniform. This can be proved as another almost immediate corollary
of Lemma \ref{averagemeasure}. However, we shall give a direct 
proof instead, since this case is an easy one and the proof is short.

\begin{lemma}\label{uniformtouniform}
Let $\eta>0$ and let $m$ and $n$ be positive integers with $m\leq n^{1/4}$
and $n\geq (12/\eta)^{12}$. Let
$J$ be a random subset of $[n]$ of size $m$, let $y$ be a random
element of $[k]^{\overline{J}}$ and let $x$ be a random element of
$[k-1]^J$ (in both cases chosen uniformly). Then the total variation
distance between the resulting distribution on $(x,y)$ and the uniform 
distribution on $[k]^n$ is at most $\eta$.
\end{lemma}

\begin{proof}
Let $z$ be an element of $[k]^n$, let $X$ be the set of coordinates $i$
such that $z_i=k$ and let $r$ be the cardinality of $X$. By the proof of
Lemma \ref{balanced}, the probability that $r$ lies between 
$n/k-n^{2/3}$ and $n/k+n^{2/3}$ is at least $1-2\exp(-2n^{1/3})$,
which is at least $1-\eta/3$. Let us assume 
that $z$ has this property. Now choose $J$ and let us calculate the 
probability that $(x,y)=z$ conditional on this choice of $J$. 

If $J\cap X\ne\emptyset$, then the probability is zero. If, however, 
if $J\cap X=\emptyset$ then it is $(k-1)^{-m}k^{-(m-n)}$. The probability
that $J\cap X=\emptyset$ is $\binom{n-r}m$, which lies between 
$(1-1/k-n^{-1/3}-m/n)^m$ and $(1-1/k+n^{-1/3})^m$. A simple calculation
shows that it therefore lies between $(1-1/k)^m(1-4n^{-1/12})$ and
$(1-1/k)^m(1+4n^{-1/12})$. Therefore, the probability that $(x,y)=z$
lies between $k^{-n}(1\pm \eta/3)$.

Let $B$ be the set of all $z$ such that $r$ does \textit{not} lie between
$n/k-n^{2/3}$ and $n/k+n^{2/3}$. Then the probability that $(x,y)\in B$
is at most $1-(1-\eta/3)^2\leq2\eta/3$. Therefore, if $A$ is any subset 
of $[k]^n$ and $\delta$ is the density of $A$, the probability that 
$(x,y)\in A$ lies between $(\delta-\eta/3)(1-\eta/3)$ and 
$\delta(1+\eta/3)+2\eta/3$, which proves the lemma.
\end{proof}

\section{A dense set with no combinatorial line correlates locally with an intersection of insensitive sets} 

In this section we shall carry out the first three stages of the proof of DHJ$_k$
(corresponding to the first three stages of the sketch proofs given earlier
of the corners theorem and DHJ$_3$).

\subsection{Finding a dense diagonal}

Let $A$ be a subset of $[k]^n$ of density $\delta$. The aim of this subsection is 
to find a combinatorial subspace $V$ of $[k]^n$ with two properties. First, the 
density of $A$ inside $V$ is not much smaller than $\delta$, and second, there
are many points of $A$ in $V$ for which the variable coordinates take values
in $[k-1]$. The densities in both cases are with respect to equal-slices measure.
The second statement corresponds to the title of this subsection: this step is
analogous to finding a dense diagonal in the corners proof. However, that
proof gave us a dense structured set that was disjoint from $A$. Here, what
we get is a structured set that is dense in a subspace. This will not help us
at all unless $A$ still has density almost $\delta$ (or better) in that subspace.
Thus, there is slightly more to this step than there was in the corners proof.

\begin{lemma}\label{densediagonal}
Let $A\subset[k]^n$ be a set of uniform density $\delta$, let $0<\eta\leq\delta/4$, 
let $m\leq n^{1/4}$ and suppose that $n\geq(16k/\eta)^{12}$. Then there 
exists a pair $(J,y)$, where $J$ is a subset of $[n]$ of size $m$ and 
$y\in[k]^{\overline{J}}$, such that one of the following two possibilities
holds:

(i) the equal-slices density of $A$ in the subspace $S_{J,y}$ is at least 
$\delta+\eta$;

(ii) the equal-slices density of $A$ in $S_{J,y}$ is at least $\delta-4\eta\delta^{-1}$
and the equal-slices density of $A$ in $S'_{J,y}$ is at least $\delta/4$.
\end{lemma}

\begin{proof}
By Corollary \ref{uniformtoequalslices}, if we choose $J$ and $y$ randomly
then the expected equal-slices density of $A$ in $S_{J,y}$ is at least 
$\delta-\eta$. If the density is never more than $\delta+\eta$, then the
probability that it is less than $\delta-4\eta\delta^{-1}$ is less than $\delta/2$,
since otherwise the average would be at most
\begin{equation*}
(1-\delta/2)(\delta+\eta)+(\delta/2)(\delta-4\eta\delta^{-1})=\delta+(1-\delta/2)\eta-2\eta<\delta-\eta,
\end{equation*}
a contradiction.

By Corollary \ref{uniformtosmallerequalslices} the average density of $A$ in 
a random set $S'_{J,y}$ is at least $\delta-\eta$. Therefore, the probability that 
$A$ has density less than $\delta/4$ in $S'_{J,y}$ is at most $1-\delta/2$, since 
otherwise the average would be at most
\begin{equation*}
\delta/2+(1-\delta/2)(\delta/4)<3\delta/4\leq\delta-\eta,
\end{equation*}
another contradiction.

It follows that if (i) does not hold then with positive probability (ii) holds.
\end{proof}

What Lemma \ref{densediagonal} tells us is that either we can pass to a 
subspace and get a density increment of $\eta$, in which case we can move
to the next stage of the iteration (after passing to a further subspace to convert
this density increment into a uniform density increment), or we find a ``dense
diagonal" in a subspace in which $A$ has not lost a significant amount of 
density.

\subsection{A ``simple" locally dense set that is almost disjoint from $A$}

Let us suppose that the second possible conclusion of Lemma \ref{densediagonal}
holds (for an $\eta$ that we are free to choose later). Then we have a combinatorial
subspace $V$ of $m$ dimensions and $A$ contains many points in $V$ for which
the variable coordinates are all in $[k-1]$. For simplicity, and without loss of 
generality, let us assume that $V=[k]^m$, and let us write $A$ for $A\cap V$.
So we are given that $A$ has equal-slices density at least $\delta-\gamma$
(where $\gamma=4\eta\delta^{-1}$) and inside $[k-1]^m$ has equal-slices
density at least $\delta/4$. Let us write $B$ for $A\cap [k-1]^m$. Finally, 
if $x\in[k]^m$ and $i,j\in[k]$, let us write $x^{i\rightarrow j}$ for the sequence
that turns all the $i$s of $x$ into $j$s.

\begin{lemma}\label{disjointsimpleset}
Let $A$ be a subset of $[k]^m$ that contains no combinatorial line, and let
$B=A\cap[k-1]^m$. For each $j\leq k-1$ let $C_j$ be the set $\{x\in[k]^m:x^{k\rightarrow j}\in B\}$.
Then $C_j$ is a $jk$-insensitive set, and $A\cap C_1\cap\dots\cap C_{k-1}\subset [k-1]^m$.
\end{lemma}

\begin{proof}
Since the condition for belonging to $C_j$ depends only on $x^{k\rightarrow j}$,
it is trivial that $C_j$ is $jk$-insensitive. 

Suppose now that $x\in C_1\cap\dots\cap C_{k-1}$ and that at least one coordinate
of $x$ takes the value $k$. Let $X$ be the set of coordinates where $x=k$. Then 
if you change all the coordinates in $X$ to $j$, you end up with a point that belongs
to $A$, since $x\in C_j$. Therefore, since $A$ contains no combinatorial line, it
follows that $x$ itself does not belong to $A$. 
\end{proof}

\begin{lemma} \label{densityofC}
Let $A$, $B$ and $C_1,\dots,C_{k-1}$ be the subsets of $[k]^m$ defined in Lemma \ref{disjointsimpleset} and let $C=C_1\cap\dots\cap C_{k-1}$.
Then for every $\delta>0$ there exists $\theta>0$ such that if $B$ has equal-slices density 
at least $\delta/4$ in $[k-1]^m$, then $C\setminus[k-1]^m$ has equal-slices density at 
least $\theta$ in $[k]^m$. 
\end{lemma}

\begin{proof}
There is a one-to-one correspondence between combinatorial lines in
$B$ and points in $C\setminus[k-1]^m$. Moreover, this one-to-one
correspondence preserves equal-slices measure (for the trivial reason
that we defined the equal-slices measure on the set of combinatorial
lines in $[k-1]^m$ by treating them as points in $[k]^m$). By the
probabilistic version of DHJ$_{k-1}$ there exists $\theta=PDHJ(k-1,\delta/4)>0$
such that the equal-slices density of combinatorial lines in $B$ is
at least $\theta$. 
\end{proof}

From this lemma and Lemma \ref{degenerate} we see that 
$\nu(A\cap C)\leq(k/\theta m)\nu(C)$. (Recall that $\nu$ is the equal-slices
measure.) If $m$ is large enough, that will be 
significantly less than $\delta$. This is the sense in which $A$ is ``almost 
disjoint" from $C$.

\subsection{A ``simple" locally dense set that correlates with $A$}

\begin{lemma} \label{firstincrement}
Let $A$, $B$ and $C_1,\dots,C_{k-1}$ be the subsets of $[k]^m$ defined in Lemma \ref{disjointsimpleset}, let $C=C_1\cap\dots\cap C_{k-1}$, and suppose that
$C$ has density $\theta$. Let $0<\gamma\leq\delta/4$ and suppose also 
that $\nu(A)\geq\delta-\gamma$ and that $\nu(A\cap C)\leq(\delta/2)\nu(C)$. 
Then there exist sets $D_1,\dots,D_{k-1}$ such
that $D_i$ is $ik$-insensitive for each $i$ and such that 
$\nu(A\cap D)\geq(\delta-\gamma)\nu(D)+\delta\theta/4k$, where 
$D=D_1\cap\dots\cap D_{k-1}$.
\end{lemma}

\begin{proof}
We begin with the observation that
\begin{equation*}
[k]^m=\bigcup_{i=1}^k C_1\cap\dots\cap C_{i-1}\cap C_i^c\cap\dots\cap C_{k-1}^c,
\end{equation*}
and that this union is in fact a partition of $[k]^m$. For each $i$ let us
write $D^{(i)}$ for the set $C_1\cap\dots\cap C_{i-1}\cap C_i^c\cap\dots\cap C_{k-1}^c$.
Then $D^{(k)}=C$. From our assumptions, we know that 
\begin{align*}
\nu(A\cap(D^{(1)}\cup\dots\cup D^{(k-1)}))&\geq\delta-\gamma-(\delta/2)\nu(D^{(k)}\\
&=(\delta-\gamma)(1-\nu(D^{(k)}))+(\delta/2-\gamma)\nu(D^{(k)})\\
&\geq(\delta-\gamma)(1-\nu(D^{(k)}))+\delta\theta/4.\\
\end{align*}
Since $1-\nu(D^{(k)})=\nu(D^{(1)}\cup\dots\cup D^{(k-1)})$, it follows by averaging
that there exists $i$ such that $\nu(A\cap D^{(i)})\geq\delta-\gamma+\delta\theta/4(k-1)$.
Since both $C_i$ and $C_i^c$ are $ik$-insensitive, this proves the lemma.
\end{proof}

Now for the next part of our argument we need to use the uniform measure. 
In order to do this, we must use our measure-transfer results again. Basically,
all we do is randomly restrict to a small subspace $V$ with the uniform measure
on it and apply Lemma \ref{equalslicestouniform}, but that is not quite the whole
story since we want \textit{two} things to happen: that the relative density of
$A\cap D\cap V$ inside $D\cap V$ is still bigger than $\delta$, and also that
the relative density of $D\cap V$ inside $V$ is not too small.

\begin{lemma} \label{restrict}
Let $\beta>0$ and let $k,r$ and $m$ be positive integers such that
$r\leq\min\{\beta m/8k,\allowbreak 
\beta m/2k^2\}$. Let $A$ and $D$ be subsets of $[k]^m$ such that 
$\nu(A\cap D)\geq (\delta-\gamma)\nu(D)+3\beta$.
Then there exists a combinatorial subspace $V$ of $[k]^m$ of dimension
$r$ such that $\mu_V(D\cap V)\geq\gamma\mu(V)$ and
$\mu_V(A\cap D\cap V)\geq(\delta-\gamma)\mu_V(D\cap V)+\beta$,
where $\mu_V$ is the uniform probability measure on $V$.
\end{lemma}

\begin{proof}
Let us choose $V$ by randomly choosing a set $J\subset [m]$ of 
size $r$, randomly choosing $y\in[k]^{\overline{J}}$ using
equal-slices measure, and taking the subspace $S_{J,y}$. 
By Lemma \ref{equalslicestouniform}, the expectation of
$\mu_V(A\cap D\cap V)-(\delta-\gamma)\mu_V(D\cap V)$ is at
least $\nu(A\cap D)-\beta-(\delta-\gamma)\nu(D)-\beta$, which 
is at least $\beta$ by our assumed lower bound for $\nu(A\cap D)$.
\end{proof}

Note that the conclusion of the lemma implies that $\mu_V(D\cap V)$
is at least $\eta$. 
\medskip

Let us now put together the results of this section.

\begin{lemma}\label{firstmainlemma}
Let $\delta>0$, let $k$ be a positive integer, let $\theta=$PDHJ$(k-1,\delta/4)$,
let $\eta=\delta^2\theta/96k$, let $\beta=\delta\theta/12k$ and let
$\gamma=4\delta^{-1}\eta=\delta\theta/24k=\beta/2$. Let $n$ be a positive integer, let 
$m=\lfloor n^{1/4}\rfloor$, let $r=\lfloor \beta m/8k^2\rfloor$ and suppose 
that $n\geq(16k/\eta)^{12}$. Let $A$ be a subset of $[k]^n$ of uniform density $\delta$. 
Then either $A$ contains a combinatorial line or there is an $r$-dimensional combinatorial 
subspace $W$ of $[k]^n$ and sets $D_1,\dots,D_{k-1}\subset W$ such that $D_j$ is $jk$-insensitive for
each $j$, and such that if we set $D$ to be $D_1\cap\dots\cap D_{k-1}$, then
$\mu_W(D)\geq\gamma$ and $\mu_W(A\cap D)\geq(\delta+\gamma)\mu_W(D)$.
\end{lemma}

\begin{proof}
Let $m=\lfloor n^{1/4}\rfloor$. Then, by Lemma \ref{densediagonal}, 
either there is 
an $m$-dimensional subspace $V$ such that $\mu_V(A)\geq\delta+\eta$,
in which case we are done (since we can pass to a random $r$-dimensional
subspace of $V$ and on average we will have the same density increment)
or there is an $m$-dimensional subspace $V$
such that the equal-slices density of $A$ in $V$ is at least $\delta-4\eta\delta^{-1}$
and the equal-slices density of $A$ in $V'$ is at least $\delta/4$, where
$V'$ is the set of points in $V$ with no variable coordinate equal to $k$.

Let $B=A\cap V'$. Then Lemma \ref{densityofC} gives us a $\theta>0$ and
sets $C_1,\dots,C_{k-1}$ such that $C_i$ is $ik$-insensitive, the intersection
$C=C_1\cap\dots\cap C_{k-1}$ is such that $C\setminus V'$ has equal-slices
density at least $\theta$, and $C\setminus V'$ is disjoint from $A$. The value
of $\theta$ can be taken to be PDHJ$(k-1,\delta/4)$. 

Let $\gamma=4\eta\delta^{-1}=\beta/2$. It is easily checked that $k/\theta m\leq\delta/2$
and that $\delta\theta/4k\geq 2\gamma$.
Therefore, Lemma \ref{firstincrement} tells us that we can find sets $D_1,\dots,D_{k-1}$
such that $D_i$ is $ik$-insensitive, and such that if $D=D_1\cap\dots\cap D_{k-1}$,
then $\nu(A\cap D)\geq(\delta-\gamma)\nu(D)+\delta\theta/4k$. 

Finally, Lemma \ref{restrict} with $\beta=\delta\theta/12k$ gives us an 
$r$-dimensional subspace $W$ of $V$ such that 
$\mu_W(A\cap D\cap W)\geq(\delta-\gamma)(D\cap W)+\beta$. 
This implies that $\mu_W(A\cap D\cap W)\geq(\delta+\gamma)\mu_W(D\cap W)$
and that $\mu_W(D\cap W)\geq\gamma$, as claimed.
\end{proof} 

\section{Almost partitioning low-complexity sets into subspaces}

We have completed one of the two main stages of the proof, which corresponds
to the first three steps of the proof we sketched of the corners theorem (and also
to the first three steps of out sketch proof of DHJ$_3$). In this section we shall
carry out a task that corresponds to the next two steps. So far, we have obtained
a density increment on a dense subset $D$ of a subspace $W$. This helps us,
because $D$ is an intersection of $ik$-insensitive sets, and therefore has
low complexity, in a certain useful sense. Our job now is to show that low-complexity
sets can be almost completely partitioned into combinatorial subspaces with 
dimension tending to infinity. To prove this, we shall follow the scheme of
argument presented in Section \ref{partitioning}. (That argument was
presented for the case $k=3$, but it can be straightforwardly generalized.)

\subsection{A $1k$-insensitive set can be almost entirely partitioned into large subspaces}

We begin by proving the result for $1k$-insensitive sets, and hence for $jk$-insensitive
sets whenever $j<k$. It will then be straightforward to deduce the result for intersections
of such sets.

\begin{lemma} \label{partitioninglemma}
Let $\eta>0$, and let $d$, $m$ and $n$ be positive integers with
$m\geq$MDHJ$_{k-1}(d,\eta)$ and $n\geq \eta^{-1}m(k+d)^m$. 
Let $D$ be a $1k$-insensitive subset of $[k]^n$. Then there are disjoint combinatorial subspaces 
$V_1,\dots,V_N$, each of which has dimension $d$ and is a subset of $D$, 
such that $\mu(V_1\cup\dots\cup V_N)\geq\mu(D)-3\eta$.
\end{lemma}

\begin{proof}
Let us write a typical element of $[k]^n$ as $(x,y)$, where
$x\in[k]^m$ and $y\in[k]^{n-m}$. For each $y$ let us write $D_y$ for the 
set $\{x\in[k]^m:(x,y)\in D\}$ and $E_y$ for the set $\{x\in[k-1]^m:(x,y)\in D\}=E_y\cap[k-1]^m$. 
Then by Lemma \ref{uniformtouniform} the average density of the sets $E_y$
is at least $\gamma-\eta\geq 2\eta$. It follows that the density of $y$ such that 
$E_y$ has density at least $\eta$ (in $[k-1]^m$) is at least $\eta$. 

If $E_y$ has density at least $\eta$, then by
our assumption about $m$ it follows that it contains a $d$-dimensional 
combinatorial subspace $U'_y$ (where this means a subspace of $[k-1]^m$).
Since $D$ is $1k$-insensitive, and therefore so is $D_y$, it follows that $D_y$
contains a $d$-dimensional combinatorial subspace $U_y$ (where this 
means a subspace of $[k]^m$). 

The number of possible $d$-dimensional subspaces of $[k]^m$ is at most
$(k+d)^m$ (since we have to decide for each coordinate $i\in[m]$ whether
to give it a fixed value in $[k]$ or to put it into one of the $d$ wildcard sets),
so by the pigeonhole principle there must exist a subspace $U\subset [k]^m$ such that the set 
$T=\{y\in[k]^{n-m}:U\times\{y\}\subset D\}$ has density at least $\eta(k+d)^{-m}$.
Since $D$ is $1k$-insensitive, it follows that $T$ is also $1k$-insensitive.

The set $U\times T$ is a subset of $D$ of density at least $\eta(k+d)^{-m}$,
and it is a union of the $d$-dimensional subspaces $U\times\{y\}$ with
$y\in T$. We now remove $U\times T$ from $D$. 

The resulting set $D_1=D\setminus(U\times T)$ is not necessarily $1k$-insensitive,
but for every $x\in[k]^m$ the set $\{y:(x,y)\in D_1\}$ \textit{is} $1k$-insensitive:
this follows immediately from the fact that both $D$ and $T$ are $1k$-insensitive.
Thus, we can at least partition $[k]^n$ into subspaces inside each of which $D_1$
is $1k$-insensitive.

This gives us the basis for an inductive argument. The inductive hypothesis is
that $D_r$ is a set of density at least $2\eta$ such that for every $x\in[k]^{rm}$
the set $\{y\in[k]^{n-rm}:(x,y)\in D_r\}$ is $1k$-insensitive, and that $D\setminus D_r$
is a union of $d$-dimensional subspaces of density at least $\eta(k+d)^{-m}$. We 
have essentially just given the proof
of the inductive step, but we need to generalize the argument very slightly.

To do this, let us write a typical element of $D_r$ as $(x,y,z)$ with $x\in[k]^{rm}$,
$y\in[k]^m$ and $z\in[k]^{n-(r+1)m}$. For each $x\in[k]^{rm}$ let $(D_r)_x$ be
$\{(y,z)\in[k]^{n-rm}:(x,y,z)\in D_r\}$ and for each pair $(x,z)$, let $(E_r)_{x,z}$ be
the set $\{y\in[k-1]^m:(x,y,z)\in D_r\}$. Then the average density of the sets
$(D_r)_x$ is the density of $D_r$, which is at least $3\eta$. It follows from 
Lemma \ref{uniformtouniform} that the average density of the sets $(E_r)_{x,z}$
is at least $2\eta$, provided that $n-rm\geq(12/\eta)^{12}$. Therefore, the
density of pairs $(x,z)$ such that $(E_r)_{x,z}$ has density at least $\eta$ is
at least $\eta$. 

If $(E_r)_{x,z}$ has density at least $\eta$, then it contains a $d$-dimensional
combinatorial subspace $U'_{x,z}$, where this is a subspace of $[k-1]^m$.
Since $(D_r)_x$ is $1k$-insensitive, it follows that it also contains a 
$d$-dimensional combinatorial subspace $U_{x,z}$, where this time we
mean a subspace of $[k]^m$. By the pigeonhole principle there is a 
$d$-dimensional subspace $U\subset[k]^m$ such that the set
$T=\{(x,z)\in[k]^{rm}\times[k]^{n-(r+1)m}:\{x\}\times U\times\{z\}\subset D_r\}$ 
has density at least $\eta(k+d)^{-m}$.

Let $D_{r+1}=D_r\setminus T\times U$ (where we interpret $T\times U$ to
mean $\{(x,y,z):(x,z)\in T, y\in U\}$). Then $T\times U$ is a union of 
$d$-dimensional subspaces of density at least $\eta(k+d)^{-m}$, and 
for every $(x,y)$ the set $\{z\in[k]^{n-(r+1)m}:(x,y,z)\in D_{r+1}\}$ is 
$1k$-insensitive.

Clearly we cannot iterate this process more than $\eta^{-1}(k+d)^m$
times. Therefore, since $n\geq \eta^{-1}m(k+d)^m$, it follows that we
can write $D$ as a disjoint union of $d$-dimensional combinatorial 
subspaces and a residual set of density at most $3\eta$, as claimed.
\end{proof}

\subsection{An intersection of $jk$-insensitive sets can be almost 
entirely partitioned into large subspaces}

The main result of this subsection is a very straightforward consequence
of Lemma \ref{partitioninglemma}. Let $F$ be the function that bounds
$n$ in terms of $d$ in that lemma (and also $\eta$ and $k$, which we
shall regard as fixed): that is, $F(d)=\lceil\eta^{-1}m(k+d)^m\rceil$, where
$m=$MDHJ$_{k-1}(d,\eta)$. Let $F^{(k-1)}(d)$ denote the result
of applying $F$ to $d\ $ $k-1$ times.

\begin{lemma} \label{mainpartitionlemma}
Let $\eta>0$, and let $d$ and $n$ be positive integers such that
$n\geq F^{(k-1)}(d)$. For each
$j\in[k-1]$ let $D_j$ be a $jk$-insensitive subset of $[k]^n$ and let
$D=D_1\cap\dots\cap D_{k-1}$. Then there are disjoint combinatorial subspaces 
$V_1,\dots,V_N$, each of which has dimension $d$ and is a subset of $D$, 
such that $\mu(V_1\cup\dots\cup V_N)\geq\mu(D)-3(k-1)\eta$.
\end{lemma}

\begin{proof}
We prove the result by induction on the number of insensitive sets
in the intersection (which is not quite the same as proving it by induction
on $k$). That is, we prove by induction that if $n\geq F^{(j)}(d)$ then
the conclusion of the lemma holds for $D^{(j)}=D_1\cap\dots\cap D_j$
and with an error of at most $3j\eta$ instead of $3(k-1)\eta$.

Lemma \ref{partitioninglemma} does the case $j=1$. In general, if
we have the result for $j-1$, then let $n\geq F^{(j)}(d)=F(F^{(j-1}(d))$.
Then by Lemma \ref{partitioninglemma} we can partition $D_j$
into combinatorial subspaces $V_1,\dots,V_N$ of dimension
$F^{(j-1)}(d)$ together with a residual set of density at most
$3\eta$. The intersection of any $D_h$ with any $V_i$ is $hk$-insensitive,
and $V_i\subset D_j$, so
\begin{equation*}
D^{(j)}\cap V_i=D^{(j-1)}\cap V_i=(D_1\cap V_i)\cap\dots\cap(D_{j-1}\cap V_i)
\end{equation*}
is an intersection of insensitive sets to which we can apply the inductive
hypothesis.

That allows us to partition each $V_i$ into combinatorial subspaces
$V_{is}$ of dimension $d$ together with a residual set of relative
density (in $V_i$) at most $3(j-1)\eta$. The union of these new residual sets
has density at most $3(j-1)\eta$ in $[k]^n$ (since the subspaces $V_i$
are disjoint), so we have partitioned $D^{(j)}$ into a union of 
$d$-dimensional combinatorial subspaces together with a residual
set of density at most $3j\eta$. This completes the inductive step.
\end{proof}

\section{Completing the proof}

At this stage our argument is essentially finished. In this section we shall
spell out why our lemmas show that DHJ$_k$ follows from DHJ$_{k-1}$.
We shall begin with a qualitative argument. After that, we shall informally
discuss how the bounds we obtain for DHJ$_k$ depend on those that
we obtain for DHJ$_{k-1}$. Finally, we shall exploit the fact that we have
good bounds when $k=2$ to give a more careful analysis of the bounds
we obtain for DHJ$_3$, which turn out to be of tower type.

\subsection{Proof that DHJ$_{k-1}$ implies DHJ$_k$} \label{sec:theproof}

Let $A\subset[k]^n$ be a set of density $\delta$. Our aim will be to find
a combinatorial subspace $V$ of dimension tending to infinity with $n$ 
such that the relative density of $A\cap V$ in $V$ is at least $\delta+c$,
where $c$ depends only on $\delta$ and $k$. If we can do that, then
we will be able to apply a simple iterative argument to complete the proof.

Lemma \ref{firstmainlemma} says that either $A$ contains a combinatorial
line or we can find an $r$-dimensional subspace $W$ and subsets 
$D_1,\dots,D_{k-1}$ of $W$ such that if $D=D_1\cap\dots\cap D_{k-1}$ 
then $\mu_W(D)$ (the density of $D$ inside $W$) is at least $\gamma$ 
and $\mu_W(A\cap D)\geq(\delta+\gamma)\mu(D)$. Here, $r$ tends to
infinity with $n$ for given $\delta$ and $k$ (and increases as $\delta$
increases), and $\gamma$ is a
parameter that depends on $\delta$ and $k$ only. To be precise,
if we let $\theta=$PDHJ$(k-1,\delta/4)$, then we can take $\gamma=\delta\theta/24k$
and $r=\lfloor \delta\theta\lfloor n^{1/4}\rfloor/96k^3\rfloor$. Thus, this
step depends on the fact that DHJ$_{k-1}$ implies PDHJ$_{k-1}$.

Now apply Lemma \ref{mainpartitionlemma} with $[k]^n$ replaced by the $r$-dimensional 
subspace $W$ and with $\eta=\gamma^2/6(k-1)$. Then we can find disjoint 
combinatorial subspaces $V_1,\dots,V_N$ of $W$ such that each has 
dimension equal to the largest $d$ for which $r\geq F^{(k-1)}(d)$,
each is a subset of $D$, and $\mu_W(V_1\cup\dots\cup V_N)\geq\mu_W(D)-\gamma^2/2$.
Here $d$ depends on $\eta$ and $k$ as well as $r$ (the dependence was
suppressed in our notation for the function $F$) and tends to infinity
as $r$ tends to infinity. The function $F$ is defined in terms of the 
function MDHJ$_{k-1}$, so this step depends on the fact that
DHJ$_{k-1}$ implies MDHJ$_{k-1}$.

It follows that 
\begin{align*}
\mu_W(A\cap(V_1\cup\dots\cup V_M))&\geq(\delta+\gamma)\mu(D)-\gamma^2/2\\
&\geq(\delta+\gamma/2)\mu(D)\\
&\geq(\delta+\gamma/2)\mu_W(V_1\cup\dots\cup V_M).\\
\end{align*}
Therefore, by averaging there must be some $i$ such that 
$\mu_W(A\cap V_i)\geq(\delta+\gamma/2)\mu(V_i)$.

Since $d$, the dimension of $W_i$ tends to infinity with $r$ and $r$ tends to 
infinity with $n$, and since $\gamma$ depends on $\delta$ and $k$ only,
we have found our desired density increment on a subspace. We may now
repeat the argument. Either $A\cap V_i$ contains a combinatorial line, or
we can pass to a further subspace (with dimension tending to infinity with
$d$ and hence with $n$) inside which the relative density is at least
$\delta+\gamma$. (In fact, we can do slightly better, since we have now
replaced $\delta$ by $\delta+\gamma/2$ so the density increment at this
second stage will be better than $\gamma/2$.) Since the density of $A$
inside any subspace is always at most 1, there can be at most $2/\gamma$
iterations of this procedure before we eventually find a combinatorial line.
Since this number of iterations depends only on $\delta$ and $k$, if
the original $n$ is large enough, $A$ must have contained a combinatorial 
line.

Since DHJ$_1$ is trivial and DHJ$_2$ follows from Sperner's theorem,
the proof of the general case of DHJ is complete.

\subsection{What bound comes out of the above argument?}

Let us briefly consider how the bound that we obtain for DHJ$_k$ relates
to the bound that we obtain for DHJ$_{k-1}$. 

We note first that EDHJ$(k-1,\delta)$ is bounded above by $(16k^2/\delta)$DHJ$(k,\delta/2)$,
by Corollary \ref{dhjimpliesedhj} (but all we really care about for the purposes of this discussion is
that the two functions are of broadly similar type). Next, recall from Theorem
\ref{equalslicesdhj} that if $A\subset[k-1]^n$ has equal-slices density at least 
$\delta$, then the equal-slices density of the set of combinatorial lines in $A$ 
is at least $(\delta/9)k^{-m}$, where $m=$EDHJ$(k-1,\delta/4)$. That is, 
PDHJ$(k-1,\delta)$ is exponentially small as a function of EDHJ$(k-1,\delta)$,
and hence as a function of DHJ$(k-1,\delta)$. In particular, if DHJ$(k-1,\delta)$
is already a tower-type function, then PDHJ$(k-1,\delta)$ behaves broadly 
like the reciprocal of DHJ$(k-1,\delta)$. It follows that the 
subspace we pass to in Lemma \ref{firstmainlemma} has dimension
broadly comparable to $n/$DHJ$(k-1,\delta)$. Equivalently, if we want
to pass to an $r$-dimensional subspace then we need $n$ to be at least
$r$DHJ$(k-1,\delta)$ or so.

The next step depends on MDHJ$_{k-1}$, and this is where things get
very expensive. The proof we gave of MDHJ$_{k-1}$ yields a bound that
is obtained as follows. Define $G_{k-1}(x)$ to be $\exp(DHJ(k-1,1/x))$.
Then MDHJ$(k-1,d,\delta)$ is bounded above by $G_{k-1}^{(d)}(1/\delta)$,
where $G_{k-1}^{(d)}$ is the $d$-fold iteration of $G_{k-1}$. The function
$F$ that comes into Lemma \ref{mainpartitionlemma} is broadly comparable
to MDHJ$(k-1,d,\delta)$ (again, assuming that MDHJ$(k-1,d,\delta)$ is
at least of tower type), so $F^{(k-1)}$ is something like $G_{k-1}^{(d(k-1))}$.

This function is so much bigger than the function $r\mapsto$DHJ$(k-1,\delta)$
that we can more or less ignore the former. Therefore, if we want to end up 
with a $d$-dimensional subspace after one round of the main iteration, we
need to start with $n$ being something like the $d(k-1)$-fold iteration of 
a function that has similar behaviour to the function $G_{k-1}$ defined
above, which is pretty similar to the function $d\mapsto$DHJ$(k-1,1/d)$.
We then have to run the whole iteration $2/\gamma$ times, where
$\gamma$ is broadly comparable to DHJ$(k-1,\delta)^{-1}$. So eventually
we need $n$ to be larger than $(k-1)d$DHJ$(k-1,\delta)$ iterations of the function
$d\mapsto$DHJ$(k-1,1/d)$, which is roughly DHJ$(k-1,\delta)$ iterations.

To rephrase slightly, if we let RDHJ$_{k-1}(s)=$DHJ$(k-1,1/s)$ (the ``R"
stands for ``reciprocal" here), then RDHJ$_k(s)$ is obtained by iterating
the function RDHJ$_{k-1}$ roughly RDHJ$_{k-1}(s)$ times.

This means that as $k$ increases by 1, the function RDHJ$_k$ goes up
by one level in the Ackermann hierarchy. (It is bigger than the corresponding
level of the Ackermann function, but not in an interesting way.)

\subsection{Bounds for DHJ$_3$}

When $k=3$, we can obtain much better bounds because in this case
we have reasonable bounds for MDHJ$_{k-1}$. Let us therefore do the
analysis a little more carefully.

First, note that Theorems \ref{probabilisticsperner} and \ref{multidimsperner}
tell us that we can take PDHJ$(2,\delta)$ to be $\delta^2/2$ and 
MDHJ$(2,d,\delta)$ to be $25\delta^{-2^d}$. Therefore, returning to 
the argument given in \S\ref{sec:theproof} and setting $k=3$, we can 
take $\theta$ to be $\delta^2/32$, $\gamma=\delta^3/2304$, and
$r=\lfloor\delta^3\lfloor n^{1/4}\rfloor/41472\rfloor$.

We apply Lemma \ref{mainpartitionlemma} with $\eta=\gamma^2/6(k-1)=\delta^6/12(2304)^2$,
which is at least $\delta^6/2^{27}$. Therefore, MDHJ$(2,d,\eta)$ is at most
$25(2^{27}\delta^{-6})^{2^d}$, and if $d\geq 10$, say, then $F(d)$ can be 
bounded above by $2\uparrow\delta^{-1}\uparrow2\uparrow2d$, where 
the symbol $\uparrow$ denotes exponentiation and $x\uparrow y\uparrow z$
means $x\uparrow(y\uparrow z)$. It follows that $F^{(2)}(d)$ is at most
$2\uparrow\delta^{-1}\uparrow 2\uparrow 2\uparrow\delta^{-1}\uparrow 2\uparrow 3d$.
(The final $3$ instead of $2$ is to (over)compensate for losing a factor of 2 earlier on in the tower.)

We may therefore take $d$ to be $\delta\log^{(6)}r$, where $\log^{(6)}$ is the
six-fold iterated logarithm. In fact, the factor of $\delta$
is unduly generous, so, bearing in mind our bound for $r$ in terms of $n$, it is 
safe to take $d$ to be $(\delta/2)\log^{(6)}n$. (Strictly speaking, we need to 
assume that $n$ is sufficiently large, but if we are generous later then this
requirement will be met by a huge margin.)

The number of iterations we need is certainly no more than $2304/\delta^3$,
but we can in fact do slightly better. It takes at most $2304/\delta^2$ iterations
for the density to increase from $\delta$ to $2\delta$. Therefore, the total number
of iterations is at most $2304\delta^{-2}(1+1/4+1/16+\dots)=3072\delta^{-2}$.
It follows that DHJ$(3,\delta)$ is bounded above by a tower of $2$s of height 
$20000\delta^{-2}$. (Since $20000>6\times 3072$, the dimension of the
space will still be vast when the iterations come to an end.) This proves
the estimate claimed in Theorem \ref{thm:our-dhj}.

\bibliographystyle{alpha}
\bibliography{dhj}

\end{document}